\RequirePackage{fix-cm}
\documentclass[smallextended]{svjour3}       
\smartqed  
\usepackage{graphicx}
\usepackage{dsfont,latexsym}
\usepackage{mathrsfs} 

\usepackage{amssymb,amscd,amsxtra}

\usepackage{mathrsfs}
\usepackage{cite}

\usepackage{color}
\definecolor{citation}{rgb}{0.2,0.58,0.2} 
\definecolor{formula}{rgb}{0.1,0.2,0.6}
\definecolor{url}{rgb}{0.3,0,0.5}    

\usepackage[colorlinks=true,linkcolor=formula,urlcolor=url,citecolor=citation]{hyperref}  

\newcommand{\R}{{\mathds R}}

\newcommand{\Om}{{\Omega}}

\DeclareMathOperator*{\esssup}{ess\,sup}
\DeclareMathOperator*{\essinf}{ess\,inf}

\DeclareMathOperator*{\essliminf}{ess\,lim\,inf}
\DeclareMathOperator*{\esslimsup}{ess\,lim\,sup}

\DeclareMathOperator*{\osc}{osc}

\DeclareMathOperator*{\supp}{supp}
\DeclareMathOperator*{\dist}{dist}

   \def\XXint#1#2#3{{\setbox0=\hbox{$#1{#2#3}{\int}$}
        \vcenter{\hbox{$#2#3$}}\kern-.5\wd0}}

\def\dys{\displaystyle}

\def\dxy{\,{\rm d}x{\rm d}y}
\def\dx{\,{\rm d}x}
\def\dy{\,{\rm d}y}
\def\dt{\,{\rm d}t}

\def\vs{\vspace{0.9mm}}

\def\eps{\varepsilon}

\allowdisplaybreaks
\makeatletter
\DeclareRobustCommand*{\bfseries}{%
  \not@math@alphabet\bfseries\mathbf
  \fontseries\bfdefault\selectfont
  \boldmath
}
\makeatother

\def\mean#1{\mathchoice%
          {\mathop{\kern 0.2em\vrule width 0.6em height 0.69678ex depth -0.58065ex
                  \kern -0.8em \intop}\nolimits_{\kern -0.4em#1}}%
          {\mathop{\kern 0.1em\vrule width 0.5em height 0.69678ex depth -0.60387ex
                  \kern -0.6em \intop}\nolimits_{#1}}%
          {\mathop{\kern 0.1em\vrule width 0.5em height 0.69678ex
              depth -0.60387ex
                  \kern -0.6em \intop}\nolimits_{#1}}%
          {\mathop{\kern 0.1em\vrule width 0.5em height 0.69678ex depth -0.60387ex
                  \kern -0.6em \intop}\nolimits_{#1}}}
           
\newlength{\defbaselineskip}
\newcommand{\setlinespacing}[1]
           {\setlength{\baselineskip}{#1 \defbaselineskip}}

\begin{document}

\title{Fractional superharmonic functions  and the Perron method   for nonlinear integro-differential equations\thanks{\it The first author has been supported by the Magnus Ehrnrooth Foundation (grant no. ma2014n1, ma2015n3). The second author has been supported by the Academy of Finland. The third author is a member of Gruppo Nazionale per l'Analisi Matematica, la Probabilit\`a e le loro Applicazioni (GNAMPA) of Istituto Nazionale di Alta Matematica ``F.~Severi'' (INdAM), whose support is acknowledged.\vspace{1mm}}
}

\titlerunning{Nonlinear integro-differential equations}

\author{Janne Korvenp\"a\"a       \and
        Tuomo Kuusi 
        \and Giampiero Palatucci
}

\authorrunning{J. Korvenp\"a\"a, T. Kuusi, G. Palatucci} 

\institute{J. Korvenp\"a\"a, T. Kuusi \at
             Department of Mathematics and Systems Analysis, Aalto University\\
             P.O. Box 1100\\
             00076 Aalto, Finland\\
              Telefax:  +358 9 863 2048\\
              \email{janne.korvenpaa@aalto.fi} \\  \email{tuomo.kuusi@aalto.fi}         
           \and
           G. Palatucci \at
              Dipartimento di Matematica e Informatica, Universit\`a degli Studi di Parma\\
              Campus - Parco Area delle Scienze,~53/A\\
              I-43124 Parma, Italy\\
              Tel: +39 521 90 21 11 
              \and
              Laboratoire MIPA, Universit\'e de N\^imes\\
              Site des Carmes - Place Gabriel P\'eri\\
              F-30021 N\^imes, France\\
              Tel: +33 466 27 95 57
              \\ \email{giampiero.palatucci@unipr.it}
}

\date{      }

\maketitle

\setcounter{tocdepth}{2}

%
%
%
%
%
\vspace{-2.8cm}
\begin{center}
 \rule{11.3cm}{0.5pt}\\[0.1cm] 
{\sc {\small To appear in}}\, {\it Math. Ann.}
\rule{11.3cm}{0.2pt}
\end{center}
\vspace{1mm}

\begin{abstract}
We deal with a class of equations driven by nonlocal, possibly degenerate, integro-differential operators of differentiability order $s\in (0,1)$ and summability growth $p>1$, whose model is the fractional $p$-Laplacian with measurable coefficients. 
We state and prove several results for the corresponding weak supersolutions, as comparison principles, a priori bounds, lower semicontinuity, and many others.
We then discuss the good definition of $(s,p)$-superharmonic functions, by also proving some related properties.
We finally introduce the nonlocal counterpart of the celebrated Perron method in nonlinear Potential Theory.
\end{abstract}

\tableofcontents

\setlinespacing{1.08}


\section{Introduction}

The Perron method (also known as the PWB method, after Perron, Wiener, and Brelot) is a consolidated method introduced at the beginning of the last century in order to solve the Dirichlet problem for the Laplace equation in a given open set~$\Omega$ with arbitrary boundary data~$g$; that is,
\begin{equation}\label{pb_dirichlet}
\begin{cases}
\mathcal{L} u = 0 & \text{in} \ \Omega \\[0.5ex]
u = g & \text{on the boundary of} \ \Omega,
\end{cases}
\end{equation}
when $\mathcal{L}=\Delta$. 
Roughly speaking, the Perron method works by finding the least superharmonic function with boundary values above the given values $g$. Under an assumption $g \in H^1(\Omega)$, the so-called Perron solution coincides with the desired Dirichlet energy solution. However, for general $g$ energy methods do not work and this is precisely the motivation of the Perron method. The method works essentially for many other partial differential equations whenever a comparison principle is available and appropriate barriers can be constructed to assume the boundary conditions. Thus, perhaps surprisingly, it turns out that the method extends to the case when the Laplacian operator in~\eqref{pb_dirichlet} is replaced by the $p$-Laplacian operator $(-\Delta_p)$ (see e.g. \cite{GLM86}) or even by
more general nonlinear operators. Consequently, the Perron method has become a fundamental tool in nonlinear Potential Theory, as well as 
in the study of several branches of Mathematics and Mathematical Physics when problems as in~\eqref{pb_dirichlet}, and the corresponding variational formulations arising from different contexts. The nonlinear Potential Theory covers a classical field having grown a lot during the last three decades from the necessity to understand better properties of supersolutions, potentials and obstacles. Much has been written about this topic and the connection with the theory of degenerate elliptic equations; we refer the reader to the exhaustive book~\cite{HKM06} by Heinonen, Kilpel\"ainen and Martio, and to the useful lecture notes~\cite{Lin06} by Lindqvist.

However -- though many important physical contexts can be surely modeled using potentials satisfying the Laplace equation or via partial differential equations as in~\eqref{pb_dirichlet} with the leading term 
 given by a  nonlinear operator as for instance the $p$-Laplacian with coefficients -- other contexts, as e.~\!g. from Biology and Financial Mathematics, are naturally described by the fractional counterpart of~\eqref{pb_dirichlet}, that is, the fractional Laplacian operator $(-\Delta)^s$. Recently,
a great attention has been focused on the study of problems involving fractional Sobolev spaces and corresponding nonlocal equations, both from a pure mathematical point of view and for concrete applications, since they naturally arise in many contexts when the interactions coming from far are determinant\footnote{For an elementary introduction to this topic and for a quite wide, but still limited, list of related references we refer to~\cite{DPV12}.}.

More in general, one can consider a class of fractional Laplacian-type operators with nonlinear growth together with a natural inhomogeneity. Accordingly, we deal 
with an extended class of nonlinear nonlocal equations, which include as a particular case some fractional Laplacian-type equations,
\begin{equation}\label{lequazione}
\mathcal{L}u(x)=\int_{\mathds{R}^n} K(x,y)|u(x)-u(y)|^{p-2}\big(u(x)-u(y)\big)\dy \, = \, 0, \quad x\in \R^n,
\end{equation}
where, for any $s\in (0,1)$ and any $p>1$,  $K$ is a suitable symmetric kernel of order $(s,p)$ with merely measurable coefficients. The integral may be singular at the origin and must be interpreted in the appropriate sense.
We immediately refer  to Section~\ref{sec_preliminaries} for the precise assumptions on the involved quantities. However, in order to simplify, one can just keep in mind the model case when the kernel $K=K(x,y)$  coincides with the Gagliardo kernel $|x-y|^{-n-sp}$; that is, when the equation in~\eqref{pb_dirichlet} reduces to
\begin{equation*}
(-\Delta)^{s}_p \,u = 0  \quad \text{in } \R^n,  
\end{equation*}
where the symbol~$\dys (-\Delta)^{s}_p$ denotes the usual {\it fractional $p$-Laplacian} operator, though in such a case the difficulties arising from having merely measurable coefficients disappear.
\vs

Let us come back to the celebrated Perron method. 
To our knowledge, especially in the nonlinear case when $p\neq 2$, the nonlocal counterpart seems basically missing\footnote{As we were finishing this manuscript, we became aware of very recent manuscript~\cite{LL16} having an independent and different approach to the problem.},  and even the theory concerning regularity and related results for the operators in~\eqref{lequazione} appears to be rather incomplete. Nonetheless, some partial results are known. It is worth citing the higher regularity contributions in~\cite{DKP15,BL15}, together with the viscosity approach in the recent paper~\cite{Lin15}, and~\cite{CLM12} for related existence and uniqueness results in the case when $p$ goes to infinity. Also, we would like to mention the related results involving measure data, as seen in~\cite{KMS15,LPPS15,BPV15}, and the fine analysis in the papers~\cite{FP14,IS14,BPS15,LL14,IMS16} where various results for fractional $p$-eigenvalues have been proven.
\vs

First, the main difference with respect to the local case is that for nonlocal equations the Dirichlet condition has to be taken in the whole complement~$\R^n\setminus\Omega$ of the domain,  
 instead of only on the boundary~$\partial\Omega$.
This comes from the very definition of the fractional operators in~\eqref{lequazione}, and it is strictly related to the natural nonlocality of those operators, and the fact that the behavior of a function outside the set~$\Omega$ does affect the problem in the whole space (and particularly on the boundary of $\Omega$),
which is indeed one of the main feature why those operators naturally arise in many contexts. On the other hand, such a nonlocal feature is also one of the main difficulties to be handled when dealing with  fractional operators. For this, some sophisticated  tools and techniques have been recently developed to treat the nonlocality, and to achieve many fundamental results for nonlocal equations.
We seize thus the opportunity to mention the breakthrough paper~\cite{Kas11} by Kassmann,  
where he revisited classical Harnack inequalities in a completely new nonlocal form by incorporating some precise nonlocal terms. This is also the case here, and indeed 
we have to consider
 a special quantity, the {\it nonlocal tail} of a function~$u$ in the ball of radius $r>0$ centered in $z \in \R^n$, given by
\begin{equation} \label{lacoda} 
{\rm Tail}(u;z,r) := \bigg(r^{sp} \int_{\R^n \setminus B_r(z)} |u(x)|^{p-1} |x-z|^{-n-sp} \dx \bigg)^{\frac{1}{p-1}}.
\end{equation}
The nonlocal tail will be a key-point in the proofs when a fine quantitative control of the long-range interactions, naturally arising when dealing with nonlocal operators as in~\eqref{lequazione}, is needed. This quantity has been introduced in~\cite{DKP15} and has been subsequently used in several recent results on the topic (see Section~\ref{sec_preliminaries} for further details).

In clear accordance with the definition in~\eqref{lacoda}, for any $p>1$ and any $s\in (0,1)$, we consider the corresponding  {\it tail space} $L^{p-1}_{sp}(\R^n)$ given by
\begin{equation} \label{def_tailspace}
L^{p-1}_{sp}(\R^n) := \Big\{ f \in L_{\rm loc}^{p-1}(\R^n)  \, : \,    {\rm Tail}(f;0,1)< \infty \Big\}.
\end{equation}
In particular, if $f \in L^{p-1}_{sp}(\R^n)$, then ${\rm Tail}(u;z,r) < \infty$ for all $z \in \R^n$ and $r\in (0,\infty)$. 
It is worth noticing that the two definitions above are very natural, by involving essentially only the leading parameters defining the nonlocal nonlinear operators; i.~\!e., their differentiability order~$s$ and their summability exponent~$p$. 
Said this, we can now approach the nonlocal counterpart of the Perron method, which, as is well-known, relies on the concept of superharmonic functions. A good definition of nonlocal nonlinear superharmonic functions is needed\footnote{
We take the liberty to call superharmonic functions appearing in this context as {\it $(s,p)$-superharmonic} emphasizing the $(s,p)$-order of the involved Gagliardo kernel. 
}.
We thus introduce the {\it $(s,p)$-superharmonic functions}, by stating and proving also their main  properties (see Section~\ref{sec_superharmonic}). The $(s,p)$-superharmonic functions constitute the nonlocal counterpart of the $p$-superharmonic functions considered in the important paper~\cite{L1}.
As expected, in view of the nonlocality of the involved operators~$\mathcal{L}$, this new definition will require to take into account the nonlocal tail in~\eqref{lacoda}, in the form of the suitable tail space~$L_{sp}^{p-1}(\R^{n})$. This is in clear accordance with the theory encountered in all the aforementioned papers, when nonlocal operators have to be dealt with in bounded domains.

\subsection{Class of $(s,p)$-superharmonic functions}
\begin{definition} \label{def_superharmonic}
We say that a function $u\colon\R^n \to [-\infty,\infty]$ is an {\it $(s,p)$-super-\break harmonic function} 
 in an open set $\Omega$ if it satisfies the following four assumptions:
\begin{itemize}
\item[(i)] $u<+\infty$ almost everywhere and $u> -\infty$ everywhere in~$\Omega$, \vs
\item[(ii)] $u$ is lower semicontinuous (l.~\!s.~\!c.) in~$\Omega$,\vs
\item[(iii)] $u$ satisfies the comparison in $\Omega$ against solutions bounded from above; that is, if $D \Subset \Omega$ is an open set and $v \in C(\overline{D})$ is a weak solution in~$D$ such that $ v_+ \in L^\infty(\R^n)$ and $u \geq v$ on $\partial D$ and almost everywhere on~$\R^n \setminus D$, then $u \geq v$ in $D$,\vs
\item[(iv)] $u_-$ belongs to $L_{sp}^{p-1}(\R^{n})$.
\end{itemize}
We say that a function $u$ is {\it $(s,p)$-subharmonic} in $\Omega$ if $-u$ is $(s,p)$-superharmonic in $\Omega$; and when both $u$ and $-u$ are $(s,p)$-superharmonic, we say that $u$ is {\it $(s,p)$-harmonic}. 
\end{definition}

\begin{remark} \label{remark:superharmonic loc bdd}
An $(s,p)$-superharmonic function
is locally bounded from below in $\Omega$ as the lower semicontinuous function attains its minimum on compact sets and it cannot be $-\infty$ by the definition.
\end{remark}

\begin{remark} \label{remark:min superharmonic}
From the definition it is immediately seen that the pointwise minimum of two $(s,p)$-superharmonic functions is $(s,p)$-superharmonic as well.
\end{remark}

\begin{remark} \label{remark:viscosity}
In the forthcoming paper~\cite{KKL16} it is shown that the class of $(s,p)$-superharmonic functions is precisely the class of viscosity supersolutions for~\eqref{lequazione} (for a more restricted class of kernels). 
\end{remark}

\begin{remark} \label{remark:harmonicity}
In Corollary~\ref{cor_harmharm} we show that a function $u$ is $(s,p)$-harmonic in $\Omega$ if and only if $u$ is a continuous weak solution in $\Omega$. 
\end{remark}

\begin{remark} \label{remark:riesz}
In the case $p=2$ and $K(x,y) = |x-y|^{-n-2s}$, the Riesz kernel 
$
u(x) = |x|^{2s-n}
$
is an $(s,2)$-superharmonic function in $\R^n$, but it is not a weak supersolution. It is the integrability $W_{\rm loc}^{s,2}$ that fails. 
\end{remark}

The next theorem describes the basic properties of $(s,p)$-superharmonic functions, which all seem to be necessary for the theory. 
\begin{theorem} \label{thm:superharmonic}
Suppose that $u$ is $(s,p)$-superharmonic in an open set $\Omega$. Then it has the following properties: 
\begin{itemize}
\item[(i)] {\bf Pointwise behavior}.
\begin{equation*} 
u(x)= \liminf_{y \to x} u(y) = \essliminf_{y \to x}u(y) \quad \text{for every } x\in\Omega.
\end{equation*}
\item[(ii)] {\bf Summability}. For 
\begin{equation*}  \label{e.bar qt}
\bar t := \begin{cases}
  \frac{(p-1)n}{n-sp}, & 1<p< \frac ns,   \\[1ex]
   +\infty , & p \geq \frac{n}s,
\end{cases} \qquad \bar q := \min\left\{\frac{n(p-1)}{n-s},p\right\},
\end{equation*}
and  $h \in (0,s)$, $t \in (0,\bar t)$ and $q \in (0,\bar q)$, $u \in W_{\rm loc}^{h,q}(\Omega) \cap L_{\rm loc}^{t}(\Omega) \cap L_{sp}^{p-1}(\R^n)$. \vs
\item[(iii)] {\bf Unbounded comparison}. 
 If $D \Subset \Omega$ is an open set and $v \in C(\overline{D})$ is a weak solution in $D$ such that $u \geq v$ on $\partial D$ and almost everywhere on~$\R^n \setminus D$, then $u \geq v$ in $D$.\vs
\item[(iv)] {\bf Connection to weak supersolutions}. If $u$ is locally bounded in $\Omega$ or $u \in W^{s,p}_{\rm loc}(\Omega)$,  
then it is a weak supersolution in $\Omega$.  
\end{itemize}
\end{theorem}

As the property (iv) of the Theorem above states, the $(s,p)$-superharmonic functions are very much connected to fractional weak supersolutions, which by the definition belong locally to the Sobolev space~$W^{s,p}$ (see Section~\ref{sec_preliminaries}). Consequently, we prove very general results for the {supersolutions}~$u$ to~\eqref{lequazione}, as e.~\!g. the natural comparison principle given in forthcoming Lemma~\ref{comp principle} which takes into account what happens outside~$\Om$, the lower semicontinuity of~$u$ (see Theorem~\ref{lsc representative}), the fact that the truncation of a supersolution is a supersolution as well (see Theorem~\ref{min(u,k)}), the pointwise convergence of sequences of supersolutions (Theorem~\ref{lemma:conv supersolution}). Clearly, the aforementioned results are expected, but further efforts and a somewhat new approach to the corresponding proofs are needed due to the nonlocal nonlinear framework considered here (see the observations at the beginning of Section~\ref{subs_nlnl} below).

\vs\vs
As said before, for 
 the nonlocal Perron method the $(s,p)$-superharmonic and $(s,p)$-subharmonic functions are the building blocks. We are now in a position to introduce this concept.

\subsection{Dirichlet boundary value problems}
As in the classical local framework, in order to solve the boundary value problem, we have to construct two classes of functions leading to the upper Perron solution and the lower Perron solution.
\begin{definition} [\bf Perron solutions] \label{perron sol} 
Let $\Omega$ be an open set. Assume that $g \in L^{p-1}_{sp}(\R^n)$. The upper class $\mathcal U_g$ of $g$ consists of all functions $u$ such that
\begin{itemize}
\item[(i)] $u$ is $(s,p)$-superharmonic in $\Omega$, \vs
\item[(ii)] $u$ is bounded from below in $\Omega$, \vs
\item[(iii)] $\displaystyle{\liminf_{\Omega \ni y \to x}u(y) \geq \esslimsup_{\R^n \setminus \Omega \ni y \to x}}\,g(y)$ for all $x \in \partial\Omega$,\vs
\item[(iv)] $u = g$ almost everywhere in $\R^n \setminus \Omega$.\vs
\end{itemize}
The lower class is $\mathcal L_g := \{ u \, : \, -u \in  \mathcal U_{-g}\}$.  
The function $\overline H_g := \inf\left\{u : u \in \mathcal U_g\right\}$ is the {\it upper Perron solution} with boundary datum~$g$ in $\Omega$,
where the infimum is taken pointwise in $\Omega$, and $\underline H_g := \sup\left\{u : u \in \mathcal L_g\right\}$ is the {\it lower Perron solution} with boundary datum~$g$ in $\Omega$.
\end{definition}
A few important observations are in order.
\begin{remark}
Notice that when $g$ is continuous in a vicinity of the boundary of $\Omega$, we can replace $\esslimsup_{y \to x}g(y)$ with $g(x)$ in Definition \ref{perron sol}(iii) above.
\end{remark}
\begin{remark} \label{perron general}
We could also consider more general Perron solutions by dropping the conditions (ii)--(iii) in Definition~\ref{perron sol} above.
However, in such a case it does not seem easy to exclude the possibility that the corresponding upper Perron solution is identically $-\infty$ in $\Omega$ even for simple boundary value functions such as constants. 
\end{remark}
In the case of the fractional Laplacian, we have the Poisson formula for the solution $u$ in a unit ball with boundary values $g$ as
\begin{align*} 
u(x)=c_{n,s}\left(1-|x|^2\right)^s \int_{\R^n \setminus B_1(0)}g(y)\left(|y|^2-1\right)^{-s}|x-y|^{-n}\dy,
\end{align*}
for every $x \in B_1(0)$; see e.~\!g. \cite{Kas11}, and also~\cite{SV14,MRS16} for related applications, and~\cite{Dyd12} for explicit computations. Using the Poisson formula one can consider simple examples in the unit ball.
\begin{example}
Taking the function $g(x) = \big||x|^2-1\big|^{s-1}$, $g \in L^1_{2s}(\R^n)$, as boundary values in the Poisson formula above, the integral does not converge. This example suggests that in this case $\overline H_g \equiv \underline H_g \equiv +\infty$ in $B_1(0)$. The example tells that if the boundary values $g$ merely belong to $L^1_{2s}(\R^n)$, we cannot, in general, expect to find reasonable solutions.
\end{example}

\begin{example}
Let us consider the previous example with $g$ reflected to the negative side in the half space, i.~\!e. 
\begin{equation*}
g(x) :=  \begin{cases}
 \left||x|^2-1\right|^{s-1}, & x_n>0 ,   \\[0.5ex]
 0, & x_n = 0,  \\[0.5ex]
 -\left||x|^2-1\right|^{s-1}, & x_n<0.
\end{cases}
\end{equation*}
Then the ``solution'' via Poisson formula, for $x\in B_1$, is 
\begin{equation*}
u(x) =  \begin{cases}
+ \infty & x_n>0 ,   \\[0.1ex]
 0, & x_n = 0,  \\[0.1ex]
 -\infty , & x_n<0,
\end{cases}
\end{equation*}
which is suggesting that we should now have $\overline{H}_g  \equiv + \infty$ and $\underline{H}_g  \equiv - \infty$ in $B_1(0)$. In view of this example it is reasonable to conjecture that the resolutivity fails in the class $L^1_{2s}(\R^n)$. 
\end{example}

In accordance with the classical Perron theory, one can prove that the upper and lower {\it nonlocal\,} Perron solutions act in the expected order (see Lemma~\ref{perron order}), and that the boundedness of the boundary values assures that the nonlocal Perron classes are non-empty (see Lemma~\ref{l.f bnd}). Then, we prove one of the main results, which is the nonlocal counterpart of the fundamental alternative theorem for the classical nonlinear Potential Theory.
\begin{theorem} \label{thm:Perron}
The Perron solutions $\overline H_g$ and $\underline H_g$ can be either identically $+\infty$ in $\Omega$, identically $-\infty$ in $\Omega$, or $(s,p)$-harmonic in $\Omega$, respectively. 
\end{theorem}
Finally, we approach the problem of resolutivity in the nonlocal framework. We state and prove a basic, hopefully useful, existence and regularity result for the solution to the nonlocal Dirichlet boundary value problem, under suitable assumptions on the boundary values and the domain~$\Omega$ (see Theorem~\ref{thm:basic Dir}). We then show  that if there is a solution to the nonlocal Dirichlet problem then it is necessarily the nonlocal Perron solution (see Lemma~\ref{sortofresol}).

\subsection{Conclusion}\label{subs_nlnl}
As one can expect, the main issues when dealing with the wide class of operators~$\mathcal{L}$ in~\eqref{lequazione} whose kernel $K$ satisfies fractional differentiability  {\it for any} $s\in (0,1)$ and $p$-summability   {\it for any} $p>1$, lie in their very definition, which combines the typical issues given by its nonlocal feature together with the ones given by its nonlinear growth behavior; also, further efforts are needed due to the presence of merely measurable coefficient in the kernel~$K$.  As a consequence, we can make use neither of some very important results recently introduced in the nonlocal theory, as the by-now classical $s$-harmonic extension framework provided by  Caffarelli and Silvestre  in~\cite{CS07}, nor of various tools as, e.~\!g., 
the strong three-commutators estimates introduced in~\cite{DLR11,DLR11b}  to deduce the regularity of weak fractional harmonic maps (see also~\cite{Sch14}),
the strong barriers and density estimates in~\cite{PSV13,SV14}, 
the pseudo-differential commutator and energy estimates in~\cite{PP13,PP15},
 and many other successful techniques which seem not to be trivially adaptable to the nonlinear framework considered here. Increased difficulties are due to the non-Hilbertian structure of the involved fractional Sobolev spaces~$W^{s,p}$ when $p$ is different than 2. 
 
Although some of our complementary results are well-known in the linear nonlocal case, i.e. when $\mathcal{L}$ reduces to the pure fractional Laplacian operator $(-\Delta)^s$, all our proofs are new even in this case.  Indeed, since we actually deal with very general operators with measurable coefficients, we have to change the approach to the problem. As a concrete example, for instance, let us mention that the proof that the supersolutions can be chosen to be lower semicontinuous functions will  follow by a careful interpolation of the local and the nonlocal contributions via a recent supremum estimate with tail (see Theorem~\ref{thm_local}). On the contrary, in the purely fractional Laplacian case when~$p=2$, the  proof of the same result is simply based on a characterization of supersolutions somewhat similar to the super mean value formula for classical superharmonic functions (see, e.~\!g., \cite[Proposition A4]{Sal12}), which is not available in our general nonlinear nonlocal framework due to the presence of possible irregular coefficients in the kernel~$K$. While in the purely (local) case when $s=1$, for the $p$-Laplace equation, the same result is a consequence of weak Harnack estimates (see, e.~\!g., \cite[Theorem 3.51-3.63]{HKM06}).

All in all, in our opinion, the contribution in the present paper is twofold. We introduce the nonlocal counterpart of the Perron method, by also introducing the concept of $(s,p)$-superharmonic functions, and extending very general results for supersolutions to the nonlocal Dirichlet problem in~\eqref{lequazione}, hence establishing a powerful framework which could be useful for developing a complete fractional nonlinear Potential Theory; in this respect, we could already refer to the forthcoming papers~\cite{KKP15,KKP16,KKL16}, where all the machinery, and in particular the {\it good} definition of fractional superharmonic functions developed here, have been required in order to deal with the nonlocal obstacle problem as well as to investigate different notions of solutions to nonlinear fractional equations of $p$-Laplace type.
Moreover, since we derive all those results for a general class of nonlinear integro-differential operators with measurable coefficients via our approach by also taking into account the nonlocal tail contributions, 
we obtain alternative proofs that are new even in the by-now classical case of the pure fractional Laplacian operator $(-\Delta)^s$.\vs


The paper is organized as follows. Firstly, an effort has been made to keep the presentation self-contained, so that   in  Section~\ref{sec_preliminaries} we collect some preliminary observations, and very recent results for fractional weak supersolutions adapted to our framework. In Section~\ref{sec_most}, we present some independent general results to be applied here and elsewhere when dealing with nonlocal nonlinear operators (Section~\ref{sec_apriori}), and we state and prove the most essential properties of fractional weak supersolutions (Sections~\ref{sec_comparison}--\ref{sec_convergence}).
Section~\ref{sec_superharmonic} is devoted to the concept of $(s,p)$-superharmonic functions: we prove Theorem~\ref{thm:superharmonic} and other related results, by also investigating their connection to the fractional weak supersolutions. Finally, in Section~\ref{sec_perrons} we focus on the nonlocal Dirichlet boundary value problems and collect some useful tools, introducing the natural nonlocal Poisson modification (Section~\ref{sec_poisson}), as well as the nonlocal Perron method, by proving the corresponding properties and the main related results as the ones in Theorem~\ref{thm:Perron} and the resolutivity presented in forthcoming Lemma~\ref{sortofresol}; see  Section~\ref{sec_perron}.

\noindent
\\ {\it Acknowledgments}. 
This paper was partially carried out while Giampiero Palatucci was visiting the Department of Mathematics and Systems Analysis at Aalto University School of Science in Helsinki,
supported by the Academy of Finland. The authors would like to thank Professor Juha Kinnunen for the hospitality and the stimulating discussions. A special thank also to Agnese Di~Castro for her useful observations on a preliminary version of this paper. 

Finally, we would like to thank Erik Lindgren, who has kindly informed us of his paper~\cite{LL16} in collaboration with Peter Lindqvist, where they deal with
a general class of fractional Laplace equations with bounded boundary data,
in the case when the  operators~$\mathcal{L}$ in~\eqref{lequazione} does reduce to the pure fractional $p$-Laplacian~$(-\Delta)^s_p$ without coefficients.  This very relevant paper contains several important results, as a fractional Perron method and a Wiener resolutivity theorem, together with the subsequent classification of the regular points, in such a nonlinear fractional framework. It could be interesting to compare those results together with the ones presented here.

\vs\section{Preliminaries}\label{sec_preliminaries}

In this section, we state the general assumptions on the quantities we are dealing with. We keep these assumptions throughout the paper.

First of all, we recall that the class of integro-differential equations in which we are interested is the following
\begin{equation}\label{problema}
\mathcal{L}u(x)={\rm P.V.}\int_{\mathds{R}^n} K(x,y)|u(x)-u(y)|^{p-2}\big(u(x)-u(y)\big)\,{\rm d}y = 0, \quad x\in\Omega. 
\end{equation}
The nonlocal operator $\mathcal{L}$ in the display above (being read a priori in the principal value sense) is driven by its {\it kernel} $K\colon \R^n\times \R^n \to [0,\infty)$, which is a measurable function  satisfying the following property:
\begin{equation}\label{hp_k}
\Lambda^{-1} \leq K(x,y)|x-y|^{n+sp} \leq \Lambda \quad \text{for a.~\!e. } x, y \in \R^n,
\end{equation}
for some $s\in (0,1)$, $p>1$, $\Lambda \geq1$. We immediately notice that in the special case when $p=2$ and $\Lambda=1$ we recover (up to a multiplicative constant) the well-known fractional Laplacian operator~$(-\Delta)^s$.

Moreover, notice that the assumption on $K$ can be weakened as follows
\begin{equation}\label{hp_k1}
\Lambda^{-1} \leq K(x,y)|x-y|^{n+sp} \leq \Lambda \quad \text{for a.~\!e. } x, y \in \R^n  \text{ s.~\!t. } |x-y| \leq 1,
\end{equation}
\begin{equation}\label{hp_k2}
0\leq K(x,y)|x-y|^{n+\eta} \leq M \quad \text{for a.~\!e. } x, y \in \R^n \text{ s.~\!t. }  |x-y| > 1,
\end{equation}
for some $s$, $p$, $\Lambda$ as above, $\eta>0$ and $M\geq 1$, as seen, e.~\!g., in the recent series of papers by Kassmann (see for instance the more general assumptions in the breakthrough paper~\cite{Kas11}). In the same sake of generalizing, one can also consider the operator $\mathcal{L}=\mathcal{L}_{\Phi}$ defined by
\begin{equation}\label{hp_k3}
\mathcal{L}_{\Phi}u(x) = {\rm P.V.}\int_{\R^n} K(x,y)\Phi(u(x)-u(y))\dy, \quad x\in\Omega, 
\end{equation}
where the real function $\Phi$ is assumed to be continuous, satisfying $\Phi(0)=0$ together with the monotonicity property
$$
\lambda^{-1}|t|^p \leq \Phi(t)t \leq \lambda|t|^p \quad \text{for every } t \in \R\setminus \{0\},
$$
for some $\lambda>1$, and some $p$ as above (see, for instance,~\cite{KMS15}).
However, for the sake of simplicity, we will take $\Phi(t)=|t|^{p-2}t$ and we will
work under the assumption in~\eqref{hp_k}. 

\vspace{2mm}

We now call up the definition of {\it the nonlocal  tail \,{\rm{Tail}$(f; z, r)$} of a function $f$ in the ball of radius $r>0$ centered in $z\in \R^n$}. We have
\begin{equation} \label{def_tail} 
{\rm Tail}(f;z,r) := \bigg(r^{sp} \int_{\R^n \setminus B_r(z)} |f(x)|^{p-1} |x-z|^{-n-sp} \dx \bigg)^{\frac{1}{p-1}},
\end{equation}
for any function $f$ initially defined in $L^{p-1}_{\textrm{loc}}(\R^n)$. As mentioned in the introduction, this quantity will play an important role in the rest of the paper. The nonlocal tail has been introduced in~\cite{DKP15}, and, as seen subsequently in several recent papers (see e.~\!g.,~\cite{BL15,BPS15, DKP14,HRS15,KMS15,KMS15b,IS14,IMS16} and many others\footnote{
When needed, our definition of Tail can also be given in a more general way by replacing the ball~$B_r$ and the corresponding~$r^{sp}$ term by an open bounded set~$E\subset\R^n$ and its rescaled measure~$|E|^{sp/n}$, respectively. This is not the case in the present paper.
}), it has been crucial in order to control in a quantifiable way the long-range interactions which naturally appear when dealing with nonlocal operators of the type considered here in~\eqref{problema}. In the following, when the center point $z$ will be clear from the context, we shall use the shorter notation \, Tail$(f; r)\equiv$ Tail$(f; z, r)$.
In accordance with~\eqref{def_tail}, we recall the definition of the tail space~$L^{p-1}_{sp}$ given in~\eqref{def_tailspace}, and we immediately notice that one can use the following equivalent definition 
\begin{equation*} 
L^{p-1}_{sp}(\R^n) = \Big\{ f \in L_{\rm loc}^{p-1}(\R^n) \; : \;   \int_{\R^n} |f(x)|^{p-1} (1+|x|)^{-n-sp} \dx < \infty \Big\}.
\end{equation*}

As expected, one can check that $L^\infty(\R^n) \subset L^{p-1}_{sp}(\R^n)$ and $W^{s,p}(\R^n) \subset L^{p-1}_{sp}(\R^n)$, where we denoted by $W^{s,p}(\R^n)$ the usual fractional Sobolev space of differentiability order $s\in (0,1)$ and summability exponent $p\geq1$, which is defined as follows
$$
W^{s,p}(\mathds{R}^n):=\left\{ v\in L^{p}(\mathds{R}^n)\,:\, \frac{|v(x)-v(y)|}{|x-y|^{\frac np+s}}\in L^p(\mathds{R}^n\times \mathds{R}^n)\right\};
$$ 
i.~\!e., an intermediary Banach space between $L^p(\mathds{R}^n)$ and $W^{1,p}(\mathds{R}^n)$ endowed with the natural norm
\begin{align*}
\|v\|_{W^{s,p}(\mathds{R}^n)}
& :=  \|v\|_{L^p(\R^n)}
+ [v]_{{W}^{s,p}(\R^n)}
 \nonumber \\
&\,\, =
\left(\int_{\mathds{R}^n} |v|^p\,{\rm d}x \right)^{\frac1 p}+ \left(\int_{\mathds{R}^n}\int_{\mathds{R}^n}\frac{|v(x)-v(y)|^p}{|x-y|^{n+sp}}\,{\rm d}x{\rm d}y\right)^{\frac1 p}.
\end{align*}
In a similar way, it is possible to define the fractional Sobolev spaces $W^{s,p}(\Omega)$ in a domain $\Omega \subset\mathds{R}^n$. By $W_0^{s,p}(\Omega)$ we denote the closure of $C_0^\infty(\Omega)$ in $W^{s,p}(\R^n)$. Conversely, if $v \in W^{s,p}(\Omega')$ with $\Omega \Subset \Omega'$ and $v=0$ outside of $\Omega$ almost everywhere, then $v$ has a representative in $W_0^{s,p}(\Omega)$ as well. 
For the basic properties of these spaces and some related topics we refer to \cite{DPV12} and the references therein. 

\vspace{2mm}

Let us denote the positive part and the negative one of a real valued function $u$  by $u_+:=\max\{u,0\}$ and $u_-:=\max\{-u,0\}$, respectively. We are now ready to provide the definitions of sub- and supersolutions $u$ to the class of 
integro-differential problems we are interested in.
\begin{definition}\label{def_supersolution}
A function $u \in W^{s,p}_{\rm{loc}}(\Omega)$ such that $u_-$ belongs $L^{p-1}_{sp}(\R^n)$ is a {\it  fractional weak $p$-supersolution} of~$\eqref{problema}$ if
\begin{align} \label{supersolution} 
\int_{\R^n} \int_{\R^n} |u(x)-u(y)|^{p-2}\big(u(x)-u(y)\big)\big(\eta(x)-\eta(y)\big) K(x,y) \dxy \ge 0
\end{align}
for every nonnegative $\eta \in C^\infty_0(\Omega)$. A function $u$ is a {\it fractional weak  $p$-subsolution} if $-u$ is a  fractional weak $p$-supersolution, and $u$ is a {\it fractional weak $p$-solution} if it is both fractional weak  $p$-sub- and $p$-supersolution.
\end{definition}
We often suppress $p$ from notation and say simply that $u$ is a weak supersolution in $\Omega$. Above $\eta \in C^\infty_0(\Omega)$ can be replaced by $\eta \in W^{s,p}_0(D)$ with every
$D \Subset \Omega$. Furthermore, it can be extended to a $W^{s,p}$-function in the whole $\R^n$ (see, e.~\!g., Section 5 in~\cite{DPV12}). 
Let us remark that we will assume that the kernel~$K$ is symmetric, which is not restrictive, in view of the weak formulation presented in Definition~\ref{def_supersolution}, since one may always define the corresponding symmetric kernel $K_{\textrm{\tiny sym}}$ given by
$$
K_{\textrm{\tiny sym}}(x,y):=\frac1{2}\Big(K(x,y)+K(y,x)\Big).
$$ 

It is worth noticing that the summability assumption of $u_-$ belonging to the tail space $L^{p-1}_{sp}(\R^n)$ is what one expects in the  nonlocal framework considered here. This is one of the novelty with respect to the analog of the definition of supersolutions in the local case, say when $s=1$, and it is necessary since here one has to use in a precise way the definition in~\eqref{def_tail} to deal with the long-range interactions; see  Remark~\ref{rem_superdef} below, and also, the regularity estimates in the aforementioned papers~\cite{DKP14,DKP15,KKP15, KMS15}. It is also worth noticing that in Definition~\ref{def_supersolution}
 it makes no difference to assume $u \in L^{p-1}_{sp}(\R^{n})$
instead of $u_- \in L^{p-1}_{sp}(\R^{n})$, as the next lemma implies. 

\begin{lemma} \label{l.tail in control}
Let $u$ be a weak supersolution in $B_{2r}(x_0)$. Then, for $c \equiv c(n,p,s)$,
\begin{eqnarray*}  
&&\hspace{-5mm}{\rm Tail}(u ; x_0,r) \\*
&&\quad  \leq  c \left( r^{\frac{sp-1-n}{p-1}} \left[u\right]_{W^{h,p-1}(B_{r}(x_0))} + r^{-\frac{n}{p-1}}\left\| u \right\|_{L^{p-1}(B_r(x_0))}  + {\rm Tail}(u_-;x_0,r)   \right)
\end{eqnarray*}
with
\begin{equation*} \label{e.weird q}
h = \max \left\{ 0, \frac{sp-1}{p-1}\right\} < s.
\end{equation*}
In particular, if $u$ is a weak supersolution in an open set $\Omega$, then $u \in L_{sp}^{p-1}(\R^n)$.
\end{lemma}
\begin{proof}
Firstly, we write the weak formulation, for nonnegative $\phi \in C_0^\infty(B_{r/2}(x_0))$ such that $\phi \equiv 1$ in $B_{r/4}(x_0)$, with $0 \leq \phi \leq 1$ and $|\nabla \phi| \leq 8/r$. We have
\begin{align*}  
0 & \leq \int_{B_{r}(x_0)} \int_{B_r(x_0)} |u(x)-u(y)|^{p-2} \big(u(x)-u(y)\big)\big(\phi(x)- \phi(y)\big) K(x,y) \dxy
\\* & \quad + \int_{\R^n \setminus B_r(x_0)} \int_{B_{r/2}(x_0)} |u(x)-u(y)|^{p-2} \big(u(x)-u(y)\big) \phi(x)  K(x,y) \dxy 
\\* & = I_1 + I_2.
\end{align*}
The first term is easily estimated using $|\phi(x) - \phi(y)| \leq 8|x-y|/r$ as
\begin{align*}  
I_1 \leq \frac{c}{r^{\min\{sp,1\}}} \left[u\right]_{W^{h,p-1}(B_{r}(x_0))}^{p-1}
\end{align*}
In order to estimate the second term, we have 
\begin{align*}  
|u(x)-u(y)|^{p-2} \big(u(x)-u(y)\big) & \leq 2^{p-1}\big( u_+^{p-1}(x)  + u_-^{p-1}(y) \big) - u_+^{p-1}(y),
\end{align*}
and thus 
\begin{equation*}  
I_2 \leq c\, r^{-sp} \left\| u \right\|_{L^{p-1}(B_r(x_0))}^{p-1} + c\, r^{n-sp} {\rm Tail}(u_- ; x_0,r)^{p-1}  - \frac{r^{n-sp}}c {\rm Tail}(u ; x_0,r)^{p-1}.
\end{equation*}
By combining the preceding displays we get the desired estimates. The second statement plainly follows by an application of H\"older's Inequality.
\end{proof}

\begin{remark}\label{rem_superdef}
The left-hand side of the inequality in~\eqref{supersolution} is finite for every $u \in W^{s,p}_{\rm{loc}}(\Omega) \cap L^{p-1}_{sp}(\R^{n})$ and for every $\eta \in C_0^\infty(\Omega)$. Indeed, for an open set $D$ such that
$\supp \eta \subset D \Subset \Omega$, we have by H\"older's Inequality
\begin{align*}\label{remark21}
&\left| \int_{\R^n} \int_{\R^n}|u(x)-u(y)|^{p-2}\big(u(x)-u(y)\big)\big(\eta(x)-\eta(y)\big)K(x,y)\dxy \right| \nonumber\\
&\qquad\le c \int_{D} \int_{D}|u(x)-u(y)|^{p-1}|\eta(x)-\eta(y)|\frac{{\rm d}x{\rm d}y}{|x-y|^{n+sp}} \nonumber\\
&\qquad\quad + c \int_{\R^n \setminus D} \int_{\supp \eta}\big(|u(x)|^{p-1}+|u(y)|^{p-1}\big)|\eta(x)||z-y|^{-n-sp} \dxy \nonumber\\[1ex]
&\qquad\le c\,[u]_{W^{s,p}(D)}^{p-1}[\eta]_{W^{s,p}(D)} + c\,\|u\|_{L^p(D)}^{p-1}\|\eta\|_{L^p(D)} + c\,{\rm Tail}(u;z,r)^{p-1}\|\eta\|_{L^1(D)},
\end{align*}
where $r:=\dist(\supp \eta, \partial D)>0$, $z \in \supp \eta$, and $c\equiv c(n,p,s,\Lambda,r,D)$. We notice that all the terms in the right-hand side  are finite since $u, \eta \in W^{s,p}(D)$ and ${\rm Tail}(u;z,r)<\infty$.
\end{remark}

\vspace{2mm}

\subsection{Algebraic inequalities}
We next collect some elementary algebraic inequalities.  In order to simplify the notation in the weak formulation~\eqref{supersolution}, from now on we denote by
\begin{equation}\label{def_l}
L(a,b):=|a-b|^{p-2}(a-b), \quad a,b \in \R.
\end{equation}
Notice that $L(a,b)$ is increasing with respect to $a$ and decreasing with respect to $b$.

\begin{lemma}{\rm (\!\!\cite[Lemma 2.1-2.2]{KKP15}).}\label{lemmai}
Let $1<p \le 2$ and $a,b,a',b' \in \R$. Then
\begin{align}\label{lemma:1<p<2}
|L(a,b)-L(a',b')| \le 4|a-a'-b+b'|^{p-1}.  
\end{align}
Let $p \ge 2$ and $a,b,a',b' \in \R$. Then 
\begin{align*}
|L(a,b)-L(a',b)| \le c\,|a-a'|^{p-1}+c\,|a-a'||a-b|^{p-2}, 
\end{align*}
and
\begin{align*}
|L(a,b)-L(a,b')| \le c\,|b-b'|^{p-1}+c\,|b-b'||a-b|^{p-2},
\end{align*}
where $c$ depends only on $p$.
\end{lemma}
\begin{remark}\label{a-b bounds}
Finally, we would like to make the following observation. In the rest of the paper, we often use the fact that there is a constant $c>0$ depending only on $p$ such that
\begin{align*}
\frac1{c} \le \frac{\big(|a|^{p-2}a-|b|^{p-2}b\big)(a-b)}{(|a|+|b|)^{p-2}(a-b)^{2}} \le c,
\end{align*}  
when $a,b \in \R$, $a \neq b$. In particular,
\begin{align} \label{a-b positive}
\big(|a|^{p-2}a-|b|^{p-2}b\big)(a-b) \geq 0, \quad a,b \in \R.
\end{align} \label{ap-2a-bp-2b}
\end{remark}

\subsection{Some recent results on  nonlocal fractional operators} \label{sec_recent}

In this section, we recall some recent results for  fractional weak sub- and supersolutions, 
which  we adapted  to our framework for the sake of the reader; see~\cite{DKP15,DKP14,KKP15} for the related proofs.
Notice that the proofs of Theorems~\ref{lem_caccio} and~\ref{thm_local} below 
 make sense even if we assume $u \in W^{s,p}_{\rm loc}(\Omega) \cap L^{p-1}_{sp}(\R^{n})$ instead of $u \in W^{s,p}(\R^{n})$.

\vspace{1mm}

Firstly, we state a general inequality which shows that the natural extension of the Caccioppoli inequality to the nonlocal framework has to take into account a suitable  tail. For other fractional Caccioppoli-type inequalities, though not taking into account the tail contribution, see~\cite{Min07,Min11}, and also~\cite{FP14}.
\begin{theorem}[{\bf Caccioppoli estimate with tail}]{\rm(\!\!\cite[Theorem~1.4]{DKP15}).}\label{lem_caccio}
Let $u$ be a weak supersolution to~\eqref{problema}.  Then, for any $B_r\equiv B_r({z})\subset \Omega$ and any nonnegative~$\varphi\in C^\infty_0(B_r)$,  the following estimate holds true
\begin{eqnarray}\label{cacio1}
&& \dys \int_{B_r}\int_{B_r}K(x,y) |w_{-}(x)\varphi(x)-w_{-}(y)\varphi(y)|^p \,{\rm d}x{\rm d}y \nonumber \\*
&&\quad\quad \leq c\int_{B_r}\int_{B_r} {K}(x,y)
 \big(\max\{w_{-}(x),w_-(y)\}\big)^p |\varphi(x)-\varphi(y)|^p\,{\rm d}x{\rm d}y \\*
&&\quad\quad \quad + c\int_{B_r}w_{-}(x)\varphi^p(x)\,{\rm d}x \left(\sup_{y\,\in\, {\rm supp}\,\varphi}\int_{\mathds{R}^n\setminus B_r}{K}(x,y)w_{-}^{p-1}(x)\,{\rm d}x \right), 
\nonumber
\end{eqnarray}
where $w_{-}:=(u-k)_{-}$ for any $k\in \R$, $K$ is any measurable kernel satisfying~\eqref{hp_k}, and $c$ depends only on~$p$.
\end{theorem}
\begin{remark}\label{rem_caccioppoli}
We underline that the estimate in~\eqref{cacio1} holds by replacing $w_-$ with $w_+:=(u-k)_+$ in the case when $u$ is  a fractional weak subsolution.
\end{remark}

A first natural consequence is the local boundedness of  fractional weak subsolutions, as stated in the following
\begin{theorem}[\bf Local boundedness]{\rm (\!\!\cite[Theorem~1.1 and Remark 4.2]{DKP15}).}\label{thm_local}
Let $u$ be a weak subsolution to~\eqref{problema} and let $B_r\equiv B_r({z})  \subset \Omega$. Then the following estimate holds true
\begin{equation}\label{sup_estimate}
\esssup_{B_{r/2}}u \leq \delta\, {\rm Tail}(u_+;{x_0},r/2)+c\, \delta^{-\gamma} \left(\mean{B_r} u_+^p\,{\rm d}x\right)^{\frac 1p},
\end{equation}
where ${\rm Tail}(\cdot)$ is defined in~\eqref{def_tail}, $\gamma={(p-1)n}/{s p^2}$,
the real parameter~$\delta \in (0,1]$, and the constant $c$ depends only on $n$, $p$, $s$, and~$\Lambda$.
\end{theorem} 
It is worth noticing that the parameter $\delta$ in~\eqref{sup_estimate} allows a precise interpolation between the local and nonlocal terms.
Combining Theorem~\ref{lem_caccio} together with a nonlocal {\it Logarithmic-Lemma} (see \cite[Lemma~1.3]{DKP15}), one can prove that  both 
the $p$-minimizers and weak solutions enjoy oscillation estimates, which naturally yield {H\"older continuity} (see Theorem~\ref{thm_holdere}) and some natural {Harnack estimates with tail}, 
as the nonlocal weak Harnack estimate presented in Theorem~\ref{thm_weakharnack} below.
\begin{theorem}[\bf H\"older continuity]{\rm (\!\!\cite[Theorem 1.2]{DKP15}).}
\label{thm_holdere}
Let $u$ be a weak solution to~\eqref{problema}. Then $u$ is locally H\"older continuous in $\Omega$. In particular, 
there are positive constants  $\alpha$, $\alpha < sp/(p-1)$, and $c$, both depending only on $n,p,s,\Lambda$, such that if $B_{2r}(x_0) \subset \Omega$, then
\begin{equation*} 
\osc_{B_\varrho(x_0)} u \leq c \left(\frac{\varrho}{r}\right)^\alpha \left[ {\rm Tail}(u;x_0,r)+ \bigg(\mean{B_{2r}(x_0)} |u|^p\dx\bigg)^{\frac 1p} \right]
\end{equation*}
holds whenever $\varrho \in (0,r]$, where ${\rm Tail}(\cdot)$ is defined in~\eqref{def_tail}.
\end{theorem}
\begin{theorem}[\bf Nonlocal weak Harnack inequality]{\rm (\!\!\cite[Theorem 1.2]{DKP14}).}\label{thm_weakharnack} 
be a weak supersolution to~\eqref{problema} such that $u\geq 0$ in $B_R\equiv B_R({x_0})\subset\Omega$. Let 
\begin{equation} \label{e.bar t}
\bar t := \begin{cases}
  \frac{(p-1)n}{n-sp}, & 1<p< \frac ns,   \\
  +\infty , & p \geq \frac{n}s. 
\end{cases}
\end{equation}
Then the following estimate holds for any $B_{r}\equiv B_{r}({x_0})\subset B_{R/2}(x_0)$ and for any $t< \bar t$
\begin{equation*}
\left( \mean{B_r} u^t\dx \right)^{\!\frac{1}{t}} \, \leq \,
c \essinf_{B_{2r}} u 
+ c \left(\frac{r}{R}\right)^{\frac{sp}{p-1}} \text{\rm Tail}(u_-; {x_0},R),
\end{equation*}
where ${\rm Tail}(\cdot)$ is defined in~\eqref{def_tail},  
and the constant $c$ depends only on $n$, $p$, $s$, and $\Lambda$.
\end{theorem}
To be precise, the case $p \geq \frac{n}{s}$ was not treated in the proof of the weak Harnack with tail in~\cite{DKP14}, but one may deduce the result in this case by straightforward modifications.

As expected, the contribution given by the nonlocal tail has again to be considered and the result is analogous to the local case if $u$ is nonnegative in the whole~$\R^n$.

\vspace{2mm}

We finally conclude this section by recalling three results for the solution to the obstacle problem in the fractional nonlinear framework we are dealing in. First, we consider the following set of functions,
\[
\mathcal K_{g,h}(\Omega,\Omega') = \Big\{u \in W^{s,p}(\Omega') \,:\, u \geq h\, \text{ a.~\!e. in } \Omega, \, u = g\, \text{ a.~\!e. on } \R^n \setminus \Omega \Big\},
\]
where $\Omega \Subset \Omega'$ are open bounded subsets of $\R^n$, $h\colon \R^n \to [-\infty,\infty)$ is the obstacle, and $g \in W^{s,p}(\Omega') \cap L^{p-1}_{sp}(\R^n)$ determines the boundary values.
The solution $u \in \mathcal K_{g,h}(\Omega,\Omega')$ to the obstacle problem satisfies 
$$
\langle \mathcal A(u) , v -u \rangle \geq 0 \qquad \text{for all } v \in  \mathcal K_{g,h}(\Omega,\Omega'),
$$
where the functional $\mathcal A(u) $ is defined, for all $w \in \mathcal K_{g,h}(\Omega,\Omega') \cap W_0^{s,p}(\Omega)$, as
$$
\langle \mathcal A(u) , w \rangle  := \int_{\R^n} \int_{\R^n} L(u(x),u(y)) \big(w(x) - w(y)\big) K(x,y) \dxy. 
$$
The results needed here are the uniqueness, the fact that such a solution is a weak supersolution and/or a weak solution to~\eqref{problema}, and the continuity of the solution up to the boundary
under precise assumptions on the functions $g$, $h$ and the set $\Omega$. 
\begin{theorem}[\bf Solution to the nonlocal obstacle problem]{\rm (\!\!\cite[Theorem 1]{KKP15}).} \label{obst prob sol} 
There exists a unique solution to the obstacle problem in $\mathcal K_{g,h}(\Omega,\Omega')$. Moreover, the solution to the obstacle problem is a weak supersolution to~\eqref{problema} in $\Omega$. 
\end{theorem}
\begin{corollary}{\rm (\!\!\cite[Corollary 1]{KKP15}).} \label{obst prob free}
Let $u$ be the solution to the obstacle problem in $\mathcal K_{g,h}(\Omega,\Omega')$. If $B_r \subset \Omega$ is such that
\[
\essinf_{B_r} (u - h) >0,
\]
then $u$ is a weak solution to~\eqref{problema} in $B_r$. In particular, if $u$ is lower semicontinuous and $h$ is upper semicontinuous in $\Omega$, then $u$ is a weak solution to~\eqref{problema} in $\Omega_+:=\big\{x \in \Omega : u(x)>h(x)\big\}$.
\end{corollary}

We say that a set $E \subset \R^n$ satisfies a measure density condition if there exist $r_0>0$ and $\delta_E \in (0,1)$ such that
\begin{equation} \label{eq:dens cond}
\inf_{0<r<r_0} \frac{| E \cap B_r(x_0)|}{|B_r(x_0)|} \geq \delta_E
\end{equation}
for every $x_0\in\partial E$. Notice that if $D$ and $\Omega$ are open sets such that $D \Subset \Omega$, there always exists an open set $U$ such that $D \Subset U \Subset \Omega$ with $\R^n \setminus U$ satisfying the measure density condition \eqref{eq:dens cond}.

\begin{theorem}{\rm (\!\!\cite[Theorem 9]{KKP15})} \label{thm:cont up to bdry}
Suppose that $\R^n \setminus \Omega$ satisfies the measure density condition \eqref{eq:dens cond} and suppose that $g \in \mathcal K_{g,h}(\Omega,\Omega')$. Let $u$ solve the obstacle problem in $\mathcal K_{g,h}(\Omega,\Omega')$.
If $g$ is continuous in $\Omega'$ and $h$ is either continuous in $\Omega$ or $h \equiv -\infty$, then $u$ is continuous in $\Omega'$.
\end{theorem}

\begin{remark} \label{remark:unifcont}
The proof of~\cite[Theorem 9]{KKP15}) gives a uniform modulus of continuity, because it is based on a priori estimates. In particular, if we have a sequence of boundary data $\{g_j\} $, $g_j \in \mathcal K_{g,h}(\Omega,\Omega')$, having a uniform modulus of continuity on compact subsets of $\Omega'$, then the corresponding family of solutions $\{u_j\}$ has a uniform modulus of continuity on compacts as well.
\end{remark}

\vs


\vs\section{Properties of the fractional weak supersolutions}\label{sec_most}

In order to prove all the main results in the present manuscript and to develop the basis for the fractional nonlinear Potential Theory, we need to perform careful computations on the strongly nonlocal form of the operators~$\mathcal{L}$ in~\eqref{problema}. Hence, it was important for us to understand how to modify the classical techniques in order to deal with nonlocal integro-differential energies, in particular to manage the contributions coming from far.
Therefore, in this section we state and prove some general and independent results for fractional weak supersolutions, to be applied here in the rest of the paper. We provide the boundedness from below and some precise control from above of the fractional energy of weak supersolutions, which could have their own interest in the analysis of equations involving the (nonlinear) fractional Laplacian and related nonlinear integro-differential operators. Next, we devote our attention to the essential properties of the weak fractional supersolutions, by investigating natural comparison principles, and lower semicontinuity. We then discuss the pointwise convergence of sequences of supersolutions and other related results.
Our results aim at constituting the fractional counterpart of the basis of the classical nonlinear Potential Theory. 

\subsection{A priori bounds for weak supersolutions}\label{sec_apriori}

The next result states that weak supersolutions are locally essentially bounded from below. 

\begin{lemma} \label{hM bound}
Let $v$ be a weak supersolution in $\Omega$, let $h \in L^{p-1}_{sp}(\R^n)$ and assume that $h \leq v \leq 0$ almost everywhere in $\R^{n}$.
Then, for all $D \Subset \Omega$ there is a constant $C \equiv C(n,p,s,\Lambda,\Omega,D,h)$ such
that
\[
\essinf_{D}v \geq -C.
\]
\end{lemma}
\begin{proof}
Let $B_{2r}(x_0) \subset \Omega$. Let $1 \le \sigma' < \sigma \le 2$ and $\rho=(\sigma-\sigma')r/2$. Then $B_{2\rho}(z) \subset B_{\sigma r}(x_0) \subset \Omega$ for a point $z \in B_{\sigma'r}(x_0)$. Thus, using the fact that $v_-=-v\geq 0$ is a weak subsolution, we can apply the estimate in Theorem~\ref{thm_local} 
choosing the interpolation parameter $\delta=1$ there. We have
\[
\esssup_{B_\rho(z)} v_- \leq {\rm Tail}(v_-;z,\rho) + c\left(\mean{B_{2\rho}(z)}v_-^{p}(x)\dx\right)^{\frac1p}.
\]
Since $h \leq v$, the tail term can be estimated as follows
\begin{align*}
{\rm Tail}(v_-;z,\rho) &\le c \left(\rho^{sp} \int_{B_{\sigma r}(x_0) \setminus B_\rho(z)} h_-^{p-1}(x)|x-z|^{-n-sp}\dx \right)^{\frac1{p-1}} \\
&\quad + c \left(\rho^{sp} \int_{\R^{n} \setminus B_{\sigma r}(x_0)} h_-^{p-1}(x)|x-z|^{-n-sp}\dx \right)^{\frac1{p-1}} \\
&\le c \left(\rho^{sp} \int_{B_{\sigma r}(x_0)} h_-^{p-1}(x)\rho^{-n-sp}\dx \right)^{\frac1{p-1}} \\
&\quad + c \left(\rho^{sp} \int_{\R^{n} \setminus B_{\sigma r}(x_0)} h_-^{p-1}(x)\left(\frac{\rho}{\sigma r}|x-x_0|\right)^{-n-sp}\dx \right)^{\frac1{p-1}} \\
&\le c\left(\sigma-\sigma'\right)^{-\frac{n}{p-1}} \Bigg[\left(\mean{B_{\sigma r}(x_0)} h_-^{p-1}(x)\dx\right)^{\frac1{p-1}}
+ {\rm Tail}(h_-;x_0,\sigma r)\Bigg] \\
&\le c\left(\sigma-\sigma'\right)^{-\frac{n}{p-1}} \Bigg[\left(\mean{B_{2 r}(x_0)} h_-^{p-1}(x)\dx\right)^{\frac1{p-1}}
+ {\rm Tail}(h_-;x_0, r)\Bigg].
\end{align*}

For the average term, in turn,
\[
\frac{|B_{\sigma r}(x_0)|}{|B_{2\rho}(z)|} = \left(\frac{\sigma r}{2\rho}\right)^{n} = \left(\frac{\sigma}{\sigma-\sigma'}\right)^{n},
\]
and thus by Young's Inequality, we obtain
\begin{align*}
\left(\mean{B_{2\rho}(z)}v_-^{p}(x)\dx\right)^{\frac1p} &\leq c\left(\sigma-\sigma'\right)^{-\frac{n}{p}} \left(\mean{B_{\sigma r}(x_0)}v_-^{p}(x)\dx\right)^{\frac1p} \\*
&\leq c\left(\esssup_{B_{\sigma r}(x_0)}v_-\right)^{\frac1p}\left(\left(\sigma-\sigma'\right)^{-n}\mean{B_{2 r}(x_0)}h_-^{p-1}(x)\dx\right)^{\frac1p} \\*
&\leq \frac12 \esssup_{B_{\sigma r}(x_0)}v_- + c \left(\sigma-\sigma'\right)^{-\frac{n}{p-1}} \left(\mean{B_{2 r}(x_0)}h_-^{p-1}(x)\dx\right)^{\frac1{p-1}}.
\end{align*}
Since the estimates above hold for every $z \in B_{\sigma'r}(x_0)$, we have after combining the estimates for tail and average terms
\begin{eqnarray*}
&& \esssup_{B_{\sigma'r}} v_- \\*
&& \qquad \leq \, \frac12\esssup_{B_{\sigma r}}v_- + c\left(\sigma-\sigma'\right)^{-\frac{n}{p-1}} \Bigg[\left(\mean{B_{2 r}} h_-^{p-1}\dx\right)^{\frac1{p-1}}
+ {\rm Tail}(h_-;x_0, r)\Bigg].
\end{eqnarray*}
Now, a standard iteration argument yields 
\begin{align*}
\esssup_{B_{r}(x_0)} v_- \leq c \Bigg[\left(\mean{B_{2 r}(x_0)} h_-^{p-1}\dx\right)^{\frac1{p-1}}
+ {\rm Tail}(h_-;x_0, r)\Bigg],
\end{align*}
which is bounded 
since $h_- \in L^{p-1}_{sp}(\R^n)$.

To finish the proof, let $D \Subset \Omega$. We can cover $D$ by finitely many balls $B_{r_i}(x_i)$, $i=1,\dots,N$, with $B_{2r_i}(x_i) \subset \Omega$, and the claim follows since
\[
\essinf_{D}v \geq -\max_{1 \leq i \leq N}\esssup_{B_{r_i}(x_i)} v_- \geq -C. 
\]
\end{proof}

From Theorem~\ref{lem_caccio} we can deduce a  Caccioppoli-type estimate as in the following
\begin{lemma} \label{lemma:cacc for bnd super}
Let $M > 0$. Suppose that $u$ is a weak supersolution in $B_{2r}\equiv B_{2r}(z)$ such that $u \leq M$ in $B_{3r/2}$.  
Then, for a positive constant  $c \equiv c(n,p,s,\Lambda)$, it holds
\begin{equation}  \label{eq:cacc for bnd super}
\int_{B_{r}} \mean{B_{r}} \frac{|u(x)-u(y)|^p}{|x-y|^{n+sp}} \dxy \leq c\,r^{- sp} H^{p}, 
\end{equation}
where
\[
H := M + \bigg( \mean{B_{3r/2}} u_-^p(x) \dx \bigg)^{\frac1p} + {\rm Tail}(u_-;z,3r/2).
\]
\end{lemma}
\begin{proof}
Let $\phi \in C_0^\infty(B_{4r/3})$ such that $0 \leq \phi \leq 1$, $\phi = 1$ in $B_r$, and $|D\phi| \leq c/r$.  
Setting $w := 2H-u$, we get
\begin{align}\label{ai0}
0 &\leq \frac1{|B_{r}|}\int_{\R^n} \int_{\R^n} L(u(x),u(y))
\big(w(x)\phi^p(x) - w(y)\phi^p(y) \big) K(x,y) \dxy \nonumber\\*[1ex]
&= - \frac1{|B_r|} \int_{B_{3r/2}} \int_{B_{3r/2} } L(w(x),w(y))\big(w(x) \phi^p(x) - w(y)\phi^p(y) \big) K(x,y) \dxy \nonumber\\* 
&\quad + \frac2{|B_r|} \int_{\R^n \setminus B_{3r/2} } \int_{B_{3r/2}} L(u(x),u(y)) w(x) \phi^p(x) {K}(x,y) \dxy \nonumber\\*[1ex]
&=: -I_1 + 2I_2.
\end{align}
Following the proof of Theorem~\ref{lem_caccio}, we can deduce, according to (3.4) in \cite{DKP15}, that
\begin{align}\label{ai1}
I_1 &\geq \frac1c \int_{B_{3r/2}} \mean{B_{3r/2}} \frac{|u(x)-u(y)|^p}{|x-y|^{n+sp}}\big(\max\big\{\phi(x),\,\phi(y)\big\}\big)^{p} \dxy \nonumber\\
&\quad  - c \int_{B_{3r/2}} \mean{B_{3r/2}} \big(2H-u(x)\big)^p \frac{|\phi(x)-\phi(y)|^p}{|x-y|^{n+sp}} \dxy \nonumber \\[1ex]
&\geq \frac1c \int_{B_{r}} \mean{B_{r}} \frac{|u(x)-u(y)|^p}{|x-y|^{n+sp}} \dxy 
- c\,r^{-sp} H^p.
\end{align}
Furthermore,
\begin{align}\label{ai2}
I_2 &\leq c \int_{\R^{n} \setminus B_{3r/2}} \mean{B_{4r/3}}\big(u(x)-u(y)\big)_+^{p-1}\big(2H-u(x)\big)|x-y|^{-n-sp} \dxy \nonumber\\
&\leq c \int_{\R^{n} \setminus B_{3r/2}} \mean{B_{4r/3}}\big(H^{p-1}+u_-^{p-1}(y)\big)\big(2H+u_-(x)\big)|y-z|^{-n-sp} \dxy \nonumber\\
&\leq c\,r^{-sp} H^p + c\,H \int_{\R^n \setminus B_{3r/2}} u_-^{p-1}(y) |y-z|^{-n-sp} \dy \nonumber\\
&\leq c\,r^{-sp} H^p,
\end{align}
where, in particular, we used Jensen's Inequality to estimate
$$
\mean{B_{4r/3}} u_-(x) \dx \leq \bigg( \mean{B_{4r/3}} u_-^p(x) \dx\bigg)^{\frac1p} \leq c\,H.
$$
By combining~\eqref{ai0} with~\eqref{ai1} and~\eqref{ai2}, we plainly obtain the estimate in~\eqref{eq:cacc for bnd super}.
\end{proof}

Using the previous result we may prove a uniform bound in $W^{s,p}$.

\begin{lemma} \label{seminorm bound}
Let $M>0$ and let $h \in L^{p-1}_{sp}(\R^n)$ with $h \leq M$ almost everywhere in $\Omega$. Let $u$ be a weak supersolution in $\Omega$ such that $u \geq h$ almost everywhere in $\R^{n}$ and $u \leq M$ almost everywhere in $\Omega$.
Then, for all $D \Subset \Omega$ there is a constant $C \equiv C(n,p,s,\Lambda,\Omega,D,M,h)$ such
that
\begin{equation}\label{eq_gag}
\int_{D} \int_{D} \frac{|u(x)-u(y)|^p}{ |x-y|^{n+sp}} \dxy \leq C.
\end{equation}
\end{lemma}
\begin{proof}
Let $D \Subset \Omega$ and denote $d:= \dist(D,\partial \Omega)>0$. We can cover the diagonal
$\mathcal{D} := \left\{(x,y) \in D \times D \,:\, |x-y|<\frac{d}{4}\right\}$ of $D \times D$ with finitely many sets of the form $B_{d/2}(z_i) \times B_{d/2}(z_i)$, $i=1,\dots,N$, such that $B_{d}(z_i) \subset \Omega$. By Lemma~\ref{hM bound} we can assume that $u$ is essentially bounded in $D$ by a constant independent of $u$. Since $u \leq M$ is a weak supersolution in $B_d(z_i)$ and $u \geq h \in L^{p-1}_{sp}(\R^n)$, we have by Lemma~\ref{lemma:cacc for bnd super} that
\[
\int_{B_{d/2}(z_i)} \int_{B_{d/2}(z_i)} \frac{|u(x)-u(y)|^p}{|x-y|^{n+sp}} \dxy \leq C'
\]
for every $i=1,\dots,N$, where $C' \equiv C'(n,p,s,\Lambda,d,M,h)$. Thus, we can split the integral in~\eqref{eq_gag} as follows
\begin{align*}\label{gag0}
\int_{D} \int_{D} \frac{|u(x)-u(y)|^p}{ |x-y|^{n+sp}} \dxy  &\leq 
 \sum_{i=1}^{N} \int_{B_{d/2}(z_i)} \int_{B_{d/2}(z_i)} \frac{|u(x)-u(y)|^p}{|x-y|^{n+sp}} \dxy \\
& \qquad
+  \iint_{(D\times D) \setminus \mathcal D} \frac{|u(x)-u(y)|^p}{ |x-y|^{n+sp}} \dxy.
\end{align*}
Now, notice that the first term in the right-hand side of the preceding inequality is bounded from above by
\begin{equation*}\label{gag1}
\sum_{i=1}^{N} \int_{B_{d/2}(z_i)} \int_{B_{d/2}(z_i)} \frac{|u(x)-u(y)|^p}{|x-y|^{n+sp}} \dxy
\leq NC';
\end{equation*}
and the second term by
\begin{equation*}\label{gag2}
\int_{D} \int_{D} \frac{|u(x)-u(y)|^p}{(d/4)^{n+sp}} \dxy
\leq C''|D|^{2}
\end{equation*}
according to the definition of $\mathcal D$, with $C''$ independent of $u$. Combining last three displays yields~\eqref{eq_gag}.
\end{proof}

\subsection{Comparison principle for weak solutions}\label{sec_comparison}

We next prove a comparison principle for weak sub- and supersolution, which   
typically constitutes a  powerful tool, playing a fundamental role in the whole PDE theory. 

\begin{lemma}[{\bf Comparison Principle}]\label{comp principle}
Let $\Omega \Subset \Omega'$ be bounded open subsets of $\R^n$. Let $u \in W^{s,p}(\Omega')$
be a weak supersolution to~\eqref{problema} in $\Omega$, and let 
$v \in W^{s,p}(\Omega')$
be a weak subsolution to~\eqref{problema} in $\Omega$ such that $u \ge v$ almost everywhere in $\R^n \setminus \Omega$. Then $u \ge v$ almost everywhere in $\Omega$ as well. 
\end{lemma}
\begin{proof}
Consider the function $\eta := (u-v)_-$. Notice that $\eta$ is a nonnegative function in $W^{s,p}_0(\Omega)$.  
For this, we can use it as a test function in~\eqref{supersolution} for both $u, v\in W^{s,p}(\Omega')$ and, by summing up, we get
\begin{align}\label{50}
0 &\le \int_{\R^n}\int_{\R^n}|u(x)-u(y)|^{p-2}\big(u(x)-u(y)\big)\big(\eta(x)-\eta(y)\big)K(x,y)\dxy  \\
&\quad - \int_{\R^n}\int_{\R^n} |v(x)-v(y)|^{p-2}\big(v(x)-v(y)\big)\big(\eta(x)-\eta(y)\big)K(x,y)\dxy. \nonumber
\end{align}
It is now convenient to split the integrals above by partitioning the whole~$\R^n$ into separate sets comparing the values of $u$ with those of $v$, so that, from~\eqref{50} we get
\begin{align}\label{51}
0 &\leq \int_{\{u<v\}}\int_{\{u<v\}} \big(L(u(x),u(y))-L(v(x),v(y))\big)\big(\eta(x)-\eta(y)\big)K(x,y)\dxy  \\
&\quad + \int_{\{u \ge v\}}\int_{\{u<v\}} \big(L(u(x),u(y))-L(v(x),v(y))\big)\eta(x)K(x,y)\dxy  \nonumber \\
&\quad - \int_{\{u<v\}}\int_{\{u \ge v\}} \big(L(u(x),u(y))-L(v(x),v(y))\big)\eta(y)K(x,y)\dxy. \nonumber
\end{align}
The goal is now to prove that the right-hand side of the inequality above is nonpositive. In view of the very definition of $\eta$ and \eqref{a-b positive}, 
we can estimate the three terms in~\eqref{51} as follows
\begin{align}\label{52}
 [ ... ]& \leq - \int_{\{u<v\}}\int_{\{u<v\}} \big(L(u(x),u(y))-L(v(x),v(y))\big) \nonumber \\
&\qquad\qquad\qquad \times \big(u(x)-u(y)-v(x)+v(y)\big)K(x,y)\dxy \nonumber \\
&\quad + \int_{\{u \ge v\}}\int_{\{u<v\}} \big(L(v(x),v(y))-L(v(x),v(y))\big)\eta(x)K(x,y)\dxy \nonumber \\
&\quad - \int_{\{u<v\}}\int_{\{u \ge v\}} \big(L(v(x),v(y))-L(v(x),v(y))\big)\eta(y)K(x,y)\dxy\nonumber  \\[1ex]
&\leq 0. 
\end{align} 
By combining~\eqref{52} with~\eqref{51}, we deduce that all the terms in~\eqref{51} have to be equal to $0$,
which implies $\eta = 0$ almost everywhere in $\{u<v\}$, in turn giving the desired result.
\end{proof}

In particular, since the weak sub- and supersolutions belong locally to $W^{s,p}$, we get the following comparison principle.
\begin{corollary} \label{comp principle2}
Let $D \Subset \Omega$. Let $u$ 
be a weak supersolution to~\eqref{problema} in $\Omega$, and let 
$v$ 
be a weak subsolution to~\eqref{problema} in $\Omega$ such that $u \ge v$ almost everywhere in $\R^n \setminus D$. Then $u \ge v$ almost everywhere in $D$.
\end{corollary}

\subsection{Lower semicontinuity of weak supersolutions}\label{sec_lower}

Now, we give an expected lower semicontinuity result for the weak supersolutions, which, as in the classic local setting, is a fundamental object to provide other important topological tools in order to develop the entire nonlinear Potential Theory. As we can see in the proof below, we will be able to obtain such a  
property essentially via the supremum estimates given by Theorem~\ref{thm_local} performing here a careful choice of the interpolation parameter~$\delta$ in~\eqref{sup_estimate} between the local contributions and the nonlocal ones. This is a relevant difference with respect to the classical nonlinear Potential Theory, where on the contrary the lower semicontinuity is a straight consequence of weak Harnack estimates (see, e.~\!g., \cite[Theorem 3.51 and 3.63]{HKM06}).
\begin{theorem}[{\bf Lower semicontinuity of supersolutions}] \label{lsc representative}
Let $u$ be a weak supersolution in $\Omega$. Then
\[
u(x)=\essliminf_{y \to x} u(y) \qquad \text{for a.\!~e. } x \in \Omega.
\] 
In particular, $u$ has a lower semicontinuous representative.
\end{theorem}
\begin{proof}
Let $D \Subset \Omega$ and
\[
E:=\bigg\{x\in D \, : \,  \lim_{r \to 0} \mean{B_r(x)}|u(x)-u(y)| \dy = 0,\, |u(x)| < \infty \bigg\}.
\]
Then, in particular, $|D \setminus E|=0$ by Lebesgue's Theorem.
Fix $z \in E$ and $\tilde{r}>0$. We may assume $B_{2\tilde r}(z) \Subset \Omega$. 
Since $v := u(z)-u$ is a weak subsolution,  we have by Theorem~\ref{thm_local} that
\begin{align} \label{esssup v}
\esssup_{B_{r}(z)} v 
 \leq
 \delta\,{\rm Tail}(v_+;z,r)+c\,\delta^{-\gamma}\left(\mean{B_{2r}(z)} v_+^p\dx \right)^{1/p} 
\end{align}
whenever $r \leq \tilde r$ and $\delta \in (0,1]$, where ${\rm Tail}$ is defined in~\eqref{def_tail} and positive constants $\gamma$ and $c$ are both independent of $u$, $r$, $z$ and $\delta$.  
Firstly, by the triangle inequality $v_+ \leq |u(z)|+ u_-$ so that we immediately have
\[
\sup_{r \in (0,\tilde r)} {\rm Tail}(v_+;z,r) \leq  c\,|u(z)|+c \sup_{r \in (0,\tilde r)} {\rm Tail}(u_-;z,r).
\]
Also, for some constant $c$ independent of $u$, $r$ and $z$, we can write
\begin{align*}
\sup_{r \in (0,\tilde r)} {\rm Tail}(v_+;z,r) &
\leq c\,|u(z)| + c \left(\tilde r^{sp} \int_{\R^n \setminus B_{\tilde r}(z)}|u_-(x)|^{p-1}|x-z|^{-n-sp}\dx \right)^{\frac1{p-1}} \\
&\quad + c \sup_{r \in (0,\tilde r)}\left(r^{sp} \int_{B_{\tilde r}(z) \setminus B_r(z)}|u_-(x)|^{p-1}|x-z|^{-n-sp}\dx \right)^{\frac1{p-1}} \\[1ex]
&\leq c\,|u(z)|+c\,{\rm Tail}(u_-;z,\tilde r) + c\esssup_{B_{\tilde r}(z)} u_-
=:M, 
\end{align*} where $M$ is finite.
Indeed, one can use the fact that $z \in E$, that $u_-$ belongs to the tail space $L^{p-1}_{sp}(\R^n)$, and that $u$ is locally essentially bounded from below in view of Lemma~\ref{hM bound}.

Now, a key-point in the present proof does consist in taking advantage of the ductility of the estimate in~\eqref{sup_estimate}, which permits us to suitably choose the parameter~$\delta$~there in order to interpolate the contribution given by the local and nonlocal terms. For this, given $\eps>0$ we choose $\delta<{\varepsilon}/{2M}$ and thus we get
\begin{align} \label{esssup v tail}
\delta\,{\rm Tail}(v_+;z,r) < \frac{\varepsilon}{2}
\end{align}
whenever $r\in(0,\tilde r)$. 

Then we estimate the term with an integral average.
Since $z \in E$ and $u$ is locally essentially bounded from below,
\[
\mean{B_{2r}(z)} \big(u(z)-u(x)\big)^p_+\dx \le \esssup_{x\in B_{2\tilde r}(z)}\big(u(z)-u(x)\big)^{p-1}_+ \mean{B_{2r}(z)} |u(z)-u(x)|\dx \to 0
\]
as $r \to 0$.  
Thus, we can choose $r_\eps \in (0,\tilde r)$ 
such that
\begin{align} \label{esssu v average}
c\,\delta^{-\gamma}\left(\mean{B_{2r_\varepsilon}(z)} \big(u(z)-u(x)\big)^p_+\dx \right)^{1/p} < \frac{\varepsilon}{2}.
\end{align}
Combining the estimates~\eqref{esssup v},~\eqref{esssup v tail}, and~\eqref{esssu v average}, it follows
\[
\esssup_{B_{r_\varepsilon}(z)}\big(u(z)-u\big) \le \varepsilon,
\]
and consequently
\[
u(z) \leq \essinf_{B_{r_\varepsilon}(z)}u+\varepsilon = \essliminf_{y \to z}u(y) + \eps.
\]
Letting $\eps \to 0$ gives
\[
u(z) \leq  \essliminf_{y \to z}u(y).
\]

The reverse inequality will follow because $z$ is a Lebesgue point:
\[
u(z) =  \lim_{r\to 0} \mean{B_r(z)} u(x) \dx
 \geq \lim_{r \to 0} \essinf_{B_r(z)} u =  \essliminf_{y \to z}u(y),
\]
and thus the claim holds for $z \in E$.
Finally, since $D \Subset \Omega$ was arbitrary, the proof is complete.
\end{proof}

\subsection{Convergence results for weak supersolutions}\label{sec_convergence}

We begin with an elementary result showing that a truncation of a weak supersolution is still a weak supersolution. 
\begin{lemma}\label{min(u,k)} 
Suppose that $u$ is a weak supersolution in $\Omega$. Then, for $k \in \R$, $\min\{u,k\}$ is a weak supersolution in $\Omega$ as well.
\end{lemma}
\begin{proof}
Clearly $\min\{u,k\} \in W_{\rm loc}^{s,p}(\Omega) \cap L_{sp}^{p-1}(\R^n)$. Thus we only need to check that it satisfies the weak formulation. To this end, take a nonnegative test function $\phi \in C^\infty_0(\Omega)$. For any $\eps>0$ we consider the marker function~$\theta_\eps$ defined by
$$
\theta_{\varepsilon}:= 1-\min\left\{1, \frac{(u-k)_+}{\varepsilon}\right\}.
$$
We choose $\eta = \theta_{\varepsilon} \phi$ as a test function in the weak formulation of $u$. Then we get
\begin{equation*}  
0 \le \int_{\R^n}\int_{\R^n} L(u(x),u(y)) \big(\theta_{\varepsilon}(x) \phi(x)-\theta_{\varepsilon}(y) \phi(y)\big)K(x,y)\dxy ,
\end{equation*}
where we denoted by~$L$ the function defined in~\eqref{def_l}.  To estimate the integrand,  we decompose $\R^n \times \R^n$ as a union of
\begin{align*}  
E_1  & := \left\{ (x,y) \in \R^n \times \R^n \, : \, u(x) \leq k\,, \;  u(y) \leq k \right\}, \\
E_{2,\eps} & := \left\{ (x,y) \in \R^n \times \R^n \, : \, u(x) \geq k+\eps \,, \; u(y) \geq k + \eps  \right\}, \\
E_{3,\eps} & := \left\{ (x,y) \in \R^n \times \R^n \, : \, u(x) \geq k +\eps   \,, \; u(y) < k+\eps  \right\}, \\
E_{4,\eps} & := \left\{ (x,y) \in \R^n \times \R^n \, : \, u(x) < k+\eps \,, \; u(y) \geq k +\eps   \right\}, \\
E_{5,\eps} & := \left\{ (x,y) \in \R^n \times \R^n \, : \, k < u(x) < k+\eps \,, \; u(y) \leq k   \right\}, \\
E_{6,\eps} & := \left\{ (x,y) \in \R^n \times \R^n \, : \, u(x) \leq k \,, \; k < u(y) < k +\eps   \right\}, \\
E_{7,\eps} & := \left\{ (x,y) \in \R^n \times \R^n \, : \, k< u(x) < k+\eps \,, \; k< u(y) < k + \eps  \right\}.
\end{align*}
Note that on $E_1$ we have $u = \min\{u,k\}$ and $\theta_\eps =1$, whereas on $E_{2,\eps}$ the test function vanishes since $\theta_{\eps}(x) = \theta_{\eps}(y) = 0$. On the other hand, on $E_{3,\eps}$ we have that $\theta_{\varepsilon}(x) = 0$ and $L(u(x),u(y))>0$. Thus, using $\theta_{\eps}(y) \geq \chi_{\{u \leq k\}}(y) $ and $\phi(x) \geq 0$, we get
  \begin{align*}   
&\iint_{E_{3,\eps}} L(u(x),u(y))  \big(\theta_{\varepsilon}(x) \phi(x)-\theta_{\varepsilon}(y) \phi(y)\big) K(x,y)\dxy 
\\* 
& \qquad \leq - \int_{\{u \leq k\}} \int_{\{u \geq k+\eps\}} L(k, u(y)) \phi(y) K(x,y)\dxy  
\\[0.5ex] 
&\quad \stackrel{\eps \to 0}{\longrightarrow} - \int_{\{u \leq k\}} \int_{\{u \geq k\}} L(k, u(y)) \phi(y)K(x,y)\dxy    
\\[0.5ex] 
& \qquad \leq \int_{\{u \leq k\}} \int_{\{u \geq k\}} L(k, u(y)) \big(\phi(x)- \phi(y)\big)K(x,y)\dxy.  
\end{align*}
The convergence follows by the monotone convergence theorem, and the last inequality follows since $\phi$ is nonnegative. Similar reasoning holds on $E_{4,\eps} $ by exchanging the roles of $x$ and $y$.
On $E_{5,\eps}$ we have $L(u(x),u(y))>0$, $\theta_\eps(y)=1$, and $\theta_\eps(x)=1-(u(x)-k)/\eps$, giving the estimate
\begin{align*}
&L(u(x),u(y))\big(\theta_\eps(x)\phi(x)-\theta_\eps(y)\phi(y)\big) \\
&\qquad =L(u(x),u(y))\big(\phi(x)-\phi(y)\big) - L(u(x),u(y))\frac{u(x)-k}{\eps}\phi(x) \\[0.5ex]
&\qquad \leq |u(x)-u(y)|^{p-1}|\phi(x)-\phi(y)|.
\end{align*}
Thus,
\begin{align*}    
& \iint_{E_{5,\eps}} L(u(x),u(y)) \big(\theta_{\varepsilon}(x) \phi(x)-\theta_{\varepsilon}(y) \phi(y)\big) K(x,y)\dxy \\
&\qquad \leq \iint_{E_{5,\eps}} |u(x)-u(y)|^{p-1}|\phi(x)-\phi(y)| K(x,y)\dxy \ \to\  0
\end{align*}
as $\eps \to 0$ by the dominated convergence theorem since $\chi_{\{k< u <k+\eps\}} \to 0$ pointwise as $\eps \to 0$. The uniform upper bound follows from the fact that $u \in W_{\rm loc}^{s,p}(\Omega) \cap L_{sp}^{p-1}(\R^n)$ and $\phi \in C_0^\infty(\Omega)$.
Similar reasoning holds on $E_{6,\eps} $ by exchanging the roles of $x$ and $y$.

Finally, on $E_{7,\eps}$ we have $\theta_\eps=1-(u-k)/\eps$, implying  \begin{align*}
&L(u(x),u(y))\big(\theta_\eps(x)\phi(x)-\theta(y)\phi(y)\big) \\*[0.5ex]
&\qquad = -\eps^{p-1}L(\theta_\eps(x),\theta_\eps(y))\big(\theta_\eps(x)\phi(x)-\theta_\eps(y)\phi(x)+\theta_\eps(y)\phi(x)-\theta_\eps(y)\phi(y)\big) \\*[0.5ex]
&\qquad = -\eps^{p-1}|\theta_\eps(x)-\theta_\eps(y)|^{p}\phi(x)+L(u(x),u(y))\theta_\eps(y)\big(\phi(x)-\phi(y)\big) \\*[0.5ex]
&\qquad \leq |u(x)-u(y)|^{p-1}|\phi(x)-\phi(y)|
\end{align*}
since $0 \leq \theta_\eps \leq 1$. Consequently,
\begin{align*}  
& \iint_{E_{7,\eps}} L(u(x),u(y)) \big(\theta_{\varepsilon}(x) \phi(x)-\theta_{\varepsilon}(y) \phi(y)\big) K(x,y)\dxy \\*
& \qquad \leq \iint_{E_{7,\eps}} |u(x)-u(y)|^{p-1}|\phi(x)-\phi(y)| K(x,y)\dxy\ \to\  0
\end{align*}
as $\eps \to 0$ by the dominated convergence theorem. Indeed, we have that  $|u(x)-u(y)| \chi_{E_{7,\eps}} \to 0$ almost everywhere as $\eps \to 0$, and the uniform upper bound follows as in the case of $E_{5,\eps}$.  
Collecting all the cases gives the desired nonnegativeness of the weak formulation for $\min\{u,k\}$:
\begin{align*} 
 0 & \leq  \liminf_{\eps \to 0 } \int_{\R^n}\int_{\R^n} L(u(x),u(y)) \big(\theta_{\varepsilon}(x) \phi(x)-\theta_{\varepsilon}(y) \phi(y)\big)K(x,y)\dxy  \\ 
  & \leq \int_{\{u \leq k\}} \int_{\{u \leq k\}} L(u(x), u(y)) \big(\phi(x)- \phi(y)\big)K(x,y)\dxy
  \\ & \quad +  \int_{\{u \geq k\}} \int_{\{u \leq k\}} L(u(x),k) \big(\phi(x)- \phi(y)\big)K(x,y)\dxy
  \\ & \quad   + \int_{\{u \leq k\}} \int_{\{u \geq k\}} L(k,u(y)) \big(\phi(x)- \phi(y)\big)K(x,y)\dxy
  \\ & = \int_{\R^n}\int_{\R^n} L\big(\min\{u(x),k\},\min\{u(y),k\}\big) \big(\phi(x)- \phi(y)\big)K(x,y)\dxy,
\end{align*}
finishing the proof. 
\end{proof}

\begin{remark}
We could also prove that the pointwise minimum of two weak supersolutions is a weak supersolution, see \cite{KKP15b}. However, we do not state the proof here since it will immediately follow from our results for $(s,p)$-superharmonic functions in Section \ref{sec_superharmonic}.
\end{remark}

Finally, we state and prove a very general fact which assures that (pointwise) limit functions of suitably 
bounded sequences of weak supersolutions are supersolutions as well.
\begin{theorem}[{\bf Convergence of sequences of supersolutions}] \label{lemma:conv supersolution} 
Let $g \in L^{p-1}_{sp}(\R^n)$ and $h \in L^{p-1}_{sp}(\R^n)$ be such that $h \leq g$ in $\R^n$.
Let  $\{u_j\}$ be a sequence of weak supersolutions in $\Omega$ such that $h \leq u_j \leq g$ almost everywhere in $\R^n$ and $u_j$ is uniformly locally essentially bounded from above in $\Omega$. Suppose that $u_j$ converges to a function $u$ pointwise almost everywhere as $j \to \infty$. Then $u$ is a weak supersolution in $\Omega$ as well.  
\end{theorem} 
\begin{proof} Fix a nonnegative $\phi \in C_0^\infty(\Omega)$ and let $D_1 $ be an open set such that $\supp \phi \subset D_1\Subset \Omega$.
Furthermore, let $D_2$ be an open set such that $D_1 \Subset D_2 \Subset \Omega$ and take large enough $M>0$ satisfying $u_j \leq M$ almost everywhere in $D_2$.
First, from Lemma~\ref{seminorm bound} for $u_j$ we deduce that 
\[
\int_{D_1} \int_{D_1} \frac{|u_j(x)-u_j(y)|^p}{ |x-y|^{n+sp}} \dxy \leq C < \infty
\]
uniformly in $j$. Therefore, Fatou's Lemma yields that $u \in W^{s,p}(D_1)$.
Moreover, the pointwise convergence implies that $h \leq u \leq g$ a.~\!e. in~$\R^n$. Accordingly, we may rewrite as
\begin{align*}
0 &\leq \int_{\R^n} \int_{\R^n} L(u_j(x),u_j(y)) \big( \phi(x) -  \phi(y) \big) K(x,y) \dxy \\[1ex]
&= \int_{\R^n} \int_{\R^n} L(u(x),u(y)) \big( \phi(x) -  \phi(y) \big) K(x,y) \dxy \\
&\quad + \int_{\R^n} \int_{\R^n} \big( L(u_j(x),u_j(y)) - L(u(x),u(y)) \big) \big( \phi(x) - \phi(y) \big) K(x,y) \dxy.
\end{align*}
We further split the second term on the right-hand side in the display above into the following two terms, by 
using the fact that $\supp \phi \subset D_1$ will assure the needed separation to write the contribution on~$D_1 \times D_1$,
\begin{align*} 
&\int_{\R^n} \int_{\R^n} \big( L(u_j(x),u_j(y)) - L(u(x),u(y)) \big) \big( \phi(x) -  \phi(y) \big) K(x,y) \dxy \\[1ex]
&\qquad = \int_{D_1} \int_{D_1} \big( L(u_j(x),u_j(y)) - L(u(x),u(y)) \big) \big( \phi(x) -  \phi(y) \big) K(x,y) \dxy \\
&\qquad\quad + 2\int_{\R^n \setminus D_1} \int_{D_1} \big( L(u_j(x),u_j(y)) - L(u(x),u(y)) \big)  \phi(x)  {K}(x,y) \dxy \\[1ex]
&\qquad =: E_{1,j} + 2E_{2,j}.
\end{align*}
Our goal is now to show that 
\[
\lim_{j\to \infty} \left( E_{1,j} + 2E_{2,j} \right) = 0,
\]
which then proves that $u$ is a weak supersolution in $\Omega$, 
as desired.

Considering first $E_{2,j}$, we have the pointwise upper bound
\begin{align*}
&\big|L(u_j(x),u_j(y)) - L(u(x),u(y))\big| \\ 
&\qquad\leq c\,\big( g_+^{p-1}(x) + g_+^{p-1}(y) + h_-^{p-1}(x) + h_-^{p-1}(y) \big),
\end{align*}
and therefore, by the dominated convergence theorem, 
\begin{eqnarray*}
&&\lim_{j\to \infty}E_{2,j} \\*
&& \quad = \, \lim_{j\to \infty} \int_{\R^n \setminus D_1} \int_{D_1} \big( L(u_j(x),u_j(y)) - L(u(x),u(y)) \big) \phi(x) {K}(x,y) \dxy \\*
&&  \quad = \, 0.
\end{eqnarray*}

Therefore, it remains to show that $\lim_{j\to \infty}  E_{1,j} = 0$. 
To this end, denote in short
\[
\Psi_{j}(x,y) := \big( L(u_j(x),u_j(y)) - L(u(x),u(y)) \big) \big( \phi(x) -  \phi(y) \big) K(x,y),
\]
and rewrite
\begin{eqnarray*} 
&&\int_{D_1} \int_{D_1} \Psi_{j}(x,y) \dxy \\
&&\qquad\qquad = \int_{A_{j,\theta} } \int_{A_{j,\theta} } \Psi_{j}(x,y) \dxy + \iint_{(D_1 \times D_1) \setminus (A_{j,\theta} \times A_{j,\theta})} \Psi_{j}(x,y) \dxy,
\end{eqnarray*}
where we have set
\[
A_{j,\theta}  :=  \Big\{x \in D_1 \, : \,  |u_j(x)-u(x)| < \theta \Big\}.
\]
On the one hand, by H\"older's Inequality we get that
\begin{align*} 
 \iint_E \Psi_{j}(x,y) \dxy &\leq c\, \bigg( \iint_E \frac{|u_j(x)-u_j(y)|^p}{|x-y|^{n+sp}} + \frac{|u(x)-u(y)|^p}{|x-y|^{n+sp}} \dxy \bigg)^{\frac{p-1}{p}} \\
& \qquad \times \bigg( \iint_E \frac{|\phi(x)-\phi(y)|^p}{|x-y|^{n+sp}} \dxy \bigg)^{\frac1p} 
\end{align*}
whenever $E$ is a Borel set of $D_1 \times D_1$. The first integral in the right-hand side of the inequality above is uniformly bounded in $j$, since the sequence $u_j$ is equibounded in~$W^{s,p}(D_1)$ as seen in the beginning of the proof. Also, since the function~$\Phi\colon\R^{n} \times \R^n \to \R$, defined by
\[
\Phi(x,y) := \frac{|\phi(x)-\phi(y)|^p}{|x-y|^{n+sp}},
\]
belongs to $L^{1}(\R^{n} \times \R^n)$, we deduce that 
\[
\lim_{j \to  \infty} \iint_{(D_1 \times D_1) \setminus (A_{j,\theta} \times A_{j,\theta})}  \Phi(x,y) \dxy = 0,
\]
because $\big|(D_1 \times D_1) \setminus (A_{j,\theta} \times A_{j,\theta})\big| \to 0$ as $j \to \infty$ for any $\theta>0$ by the pointwise convergence of $u_j$ to $u$. 

On the other hand,
\begin{align*} 
|\Psi_j(x,y)| &\leq c\,\frac{|\phi(x)-\phi(y)|}{|x-y|^{n+sp}} \big| u_j(x)-u(x) - u_j(y) + u(y) \big| \\
&\qquad \times \int_0^1 \big| t\big( u_j(x) - u_j(y)\big) + (1-t)\big( u(x) - u(y)\big) \big|^{p-2} \dt,
\end{align*}
where we can estimate
\begin{align*}
& \big| u_j(x)-u(x) - u_j(y) + u(y) \big| \\
&\qquad \leq \big| |u_j(x){-}u(x)| + |u_j(y){-}u(y)| \big|^\sigma \big| |u_j(x){-}u_j(y)| + |u(x){-}u(y)| \big|^{1-\sigma} 
\end{align*}
for any $\sigma \in (0,1)$.

Now, we have to distinguish two cases depending on the summability exponent~$p$. 
In the case when~$p\geq 2$, we obtain in $A_{j,\theta} \times A_{j,\theta}$ that 
\[
|\Psi_j(x,y)|
\leq c\,\theta^{\sigma} \frac{\big(|u_j(x)-u_j(y)|+|u(x)-u(y)|\big)^{p-1-\sigma}}{|x-y|^{s(p-1-\sigma)}}
\frac{|\phi(x)-\phi(y)|}{|x-y|^{n+s(1+\sigma)}},
\]
and thus by H\"older's Inequality we obtain  
\[
\int_{A_{j,\theta} } \int_{A_{j,\theta} } \Psi_{j}(x,y) \dxy \leq c\,\theta^\sigma C  \left( \int_{D_1} \int_{D_1} \frac{|\phi(x)-\phi(y)|^{q}}{|x-y|^{n+s(1+\sigma) q}} \dxy \right)^{\frac1{q}},
\]
where $q := [p/(p-1-\sigma)]'=p/(1+\sigma)$ and $C$ is independent of $j$ and $\theta$. Taking
\[
\sigma=\min\left\{\frac{1-s}{2s},\frac{p-1}{2},\frac{1}{2}\right\},
\]
we finally get that 
\begin{align} \label{eq:conv prel xxx} 
\int_{A_{j,\theta} } \int_{A_{j,\theta} } \Psi_{j}(x,y) \dxy \leq \widetilde C \theta^{\sigma},
\end{align}
where $\widetilde C$ is independent of $j$ and $\theta$.

On the other hand, in the case when $1<p<2$, we obtain by~\eqref{lemma:1<p<2}
\begin{align*}
|\Psi_j(x,y)| &\leq c\, \frac{\left| u_j(x)-u(x) - u_j(y) + u(y) \right|^{p-1}}{|x-y|^{s(p-1)}} 
\frac{|\phi(x)-\phi(y)|}{|x-y|^{n+s}} \\
&\leq c\,\theta^\sigma  \frac{\left| u_j(x)-u(x) - u_j(y) + u(y) \right|^{p-1-\sigma}}{|x-y|^{s(p-1-\sigma)}} 
\frac{|\phi(x)-\phi(y)|}{|x-y|^{n+s(1+\sigma)}}
\end{align*}
in $A_{j,\theta} \times A_{j,\theta}$, and now it suffices to act as in the case~$p\geq 2$ above in order to prove the estimate in~\eqref{eq:conv prel xxx} also in such a sublinear case.

Finally, it suffices to collect all the estimates above in order to conclude that actually
\[
\lim_{j \to \infty} E_{1,j} = 0
\]
holds since $\theta$ can be chosen arbitrarily small. This finishes the proof.
\end{proof}

If the sequence is increasing, we do not have to assume any boundedness from above.

\begin{corollary} \label{lemma:conv supersolution2}
Let  $\{u_j\}$ be an increasing sequence of weak supersolutions in $\Omega$ such that $u_j$ converges to a function $u \in W^{s,p}_{\rm loc}(\Omega) \cap L^{p-1}_{sp}(\R^n)$ pointwise almost everywhere in $\R^n$ as $j \to \infty$.
Then $u$ is a weak supersolution in $\Omega$ as well.  
\end{corollary}
\begin{proof} 
For any $M>0$, denote by $u_{M}:=\min\{u,M\}$ and $u_{M,j}:=\min\{u_j,M\}$, which is a weak supersolution by Lemma~\ref{min(u,k)}. Then $\{u_{M,j}\}_j$ is a sequence satisfying the assumptions of Theorem~\ref{lemma:conv supersolution} converging pointwise almost everywhere to $u_M$, and consequently $u_M$ is a weak supersolution in $\Omega$.
Let $\eta \in C^\infty_0(\Omega)$ be a nonnegative test function.
Since
\[
\left|L(u_M(x),u_M(y))\right| \leq |u(x)-u(y)|^{p-1}
\]
for every $M>0$ and every $x,y \in \R^n$, where $u \in W^{s,p}_{\rm loc}(\Omega) \cap L^{p-1}_{sp}(\R^n)$, we can let $M \to \infty$ to obtain by the dominated convergence theorem that
\[
\int_{\R^n}\int_{\R^n}L(u(x),u(y))\big(\eta(x)-\eta(y)\big)K(x,y)\dxy \geq 0.
\]
We conclude that $u$ is a weak supersolution in $\Omega$.
\end{proof}

A similar result as Theorem \ref{lemma:conv supersolution} holds also for sequences of weak solutions.

\begin{corollary} \label{lemma:conv solution}
Let  $h,g \in L_{sp}^{p-1}(\R^n)$ and let $\{u_j\}$ be a sequence of weak solutions in $\Omega$ such that $h \leq u_j \leq g$ and $u_j \to  u$  pointwise almost everywhere in $\R^n$ as $j \to \infty$. Then $u$ is a weak solution in $\Omega$. 
\end{corollary}
\begin{proof}
Since both $u_j$ and $-u_j$ are weak supersolutions in $\Omega$, we have that $u_j$ is uniformly locally essentially bounded in $\Omega$ by Lemma~\ref{hM bound}.
Then $u$ is a weak solution in $\Omega$ since both $u$ and $-u$ are weak supersolutions by Theorem~\ref{lemma:conv supersolution}.
\end{proof}

We conclude the section with a crucial convergence result concerning continuous weak solutions. 

\begin{corollary} \label{harnack conv 0}
Let  $h,g \in L_{sp}^{p-1}(\R^n)$ and let $\{u_j\}$ be a sequence of continuous weak solutions in $\Omega$ such that $h \leq u_j \leq g$ and that $\lim_{j \to \infty} u_j$ exists almost everywhere in $\R^n$.
Then $u := \lim_{j \to \infty} u_j$ exists at every point of $\Omega$ and $u$ is a continuous weak solution in $\Omega$.  
\end{corollary}
\begin{proof}
According to Corollary~\ref{lemma:conv solution}, $u$ is a weak solution in $\Omega$. Therefore only continuity of $u$ in $\Omega$ and pointwise convergence need to be checked. Letting $B_{3r}(x_0)$ be a ball in $\Omega$, we have by Lemma~\ref{hM bound} and the uniform Tail space bounds that
\[
\sup_{j} \left(\sup_{B_{2r}(x_0)} |u_j| +{\rm Tail}(u_j;x_0,r)\right) \leq C,
\]
where $C$ is independent of $u_j$ and $u$. Using now the H\"older continuity estimate in Theorem~\ref{thm_holdere}, we see that 
\begin{align*}
\osc_{B_\rho(x_0)} u_j &\leq c\left(\frac\rho r\right)^\alpha\left(\sup_{B_{2r}(x_0) } |u_j|  + {\rm Tail}(u_j;x_0,r) \right) \leq    c\left(\frac\rho r\right)^\alpha C,
\end{align*}
where $\rho \in (0,r)$ and $\alpha \equiv \alpha(n,p,s,\Lambda) \in (0,1)$. Therefore the sequence $\{u_j\}$ is equicontinuous on compact subsets of $\Omega$, and thus the continuity of $u$ and pointwise convergence in $\Omega$ follow from the Arzel\`a--Ascoli theorem. This finishes the proof.
\end{proof}
\vs


\vs\section{$(s,p)$-superharmonic functions}\label{sec_superharmonic}

In this section, we study the nonlocal superharmonic functions for the nonlinear integro-differential equations in~\eqref{lequazione},
which we have defined in the introduction; recall Definition \ref{def_superharmonic}.
As well-known, the superharmonic functions constitute an important class of functions which have been extensively used in PDE and in classical Potential Theory, as well as in Complex Analysis. Their fractional counterpart has to take into account the nonlocality of the operators in~\eqref{problema} and thus it has to
incorporate the summability assumptions of the negative part of the functions in the tail space~$L_{sp}^{p-1}$ defined in~\eqref{def_tailspace}.

\subsection{Bounded $(s,p)$-superharmonic functions} We first move towards proving Theorem~\ref{thm:superharmonic}(iv). We begin with an elementary approximation result for lower semicontinuous functions. The proof is standard and goes via infimal convolution. However, due to the nonlocal framework we need a suitable pointwise control of approximations over $\R^n$, and hence we present the details. 
\begin{lemma} 
\label{lem_approximation}
Let $u$ be an $(s,p)$-superharmonic function in $\Omega$ and let $D \Subset \Omega$. Then there is an increasing sequence of smooth functions $\{\psi_j\}$ such that 
\begin{equation*} 
\lim_{j \to \infty} \psi_j(x) = u(x) \qquad \text{for all } x \in D.
\end{equation*}
\end{lemma}
\begin{proof}
Define the increasing sequence of continuous functions $\{\widetilde \psi_j\}$ as follows
\[
\widetilde \psi_j(x) := \min_{y \in \overline D} \Big\{\min\big\{j,\, u(y)\big\} + j^2 |x-y| \Big\} - \frac1{j}.
\]
Notice that, by the very definition, $\widetilde \psi_j(x) \leq u(x)-1/j < u(x)$ in $D$. Since $u$ is locally bounded from below,
$u(y) \ge -M$ in $\overline D$ for some $M<\infty$. Also, by the lower semicontinuity, the minimum is attained at some $y_j \in \overline D$, and thus we have 
\[
j - \frac1{j} \geq \widetilde \psi_j(x) \geq -M + j^2 |x-y_j| - \frac1{j},
\]
which yields
\[
|x-y_j| \le \frac{j+M}{j^2} =: r_j < \infty,
\]
where $r_j \to 0$  as $j \to \infty$. Since $u$ is lower semicontinuous, we have that in $D$ 
\begin{align*}
u(x) &\leq 
\lim_{j \to \infty} \left(  \inf_{y\in B_{r_j}(x)} \min\big\{j,u(y)\big\} -\frac1{j} \right)\\
&\le \lim_{j \to \infty} \left( \inf_{y\in B_{r_j}(x)} \Big\{\min\big\{j,u(y)\big\} + j^2 |x-y| \Big\} - \frac1{j}\right)
= \lim_{j \to \infty} \widetilde \psi_j(x).
\end{align*}
Hence, $\lim_{j \to \infty} \widetilde \psi_j(x) = u(x)$ for all $x \in D$.
Finally, since $\{\widetilde \psi_j\}$ is an increasing sequence of continuous functions in $D$
and
\[
\widetilde \psi_{j+1} - \widetilde \psi_{j} \geq \frac1{j} - \frac1{j+1} > 0 \qquad \text{in } D,
\]
we can find smooth functions $\psi_j$ such that 
$\widetilde \psi_{j} \leq \psi_j < \widetilde \psi_{j+1}$ in $D$. 
Now $\{\psi_j\}$ is the desired sequence of functions.
\end{proof}

Using the previous approximation lemma, we can show that the $(s,p)$-superharmonic functions can  be also approximated by continuous weak supersolutions in regular sets.

\begin{lemma} \label{lem_approximation2}
Let $u$ be an $(s,p)$-superharmonic function in $\Omega$ and let $D \Subset \Omega$ be an open set such that $\R^n \setminus D$ satisfies the measure density condition~\eqref{eq:dens cond}.
Then there is an increasing sequence $\{u_j\}$, $u_j \in C(\overline{D})$, of weak supersolutions in $D$ converging to $u$ pointwise in $\R^n$.
\end{lemma}
\begin{proof} 
Let $U$ be an open set satisfying $D  \Subset U \Subset \Omega$, which is possible by Urysohn's Lemma. By Lemma~\ref{lem_approximation}, there is an increasing sequence of smooth functions $\{\psi_j\}$, $\psi_j \in C^\infty(\overline U)$, converging to $u$ pointwise in $U$.
For each $j$, define
\[
g_j(x):=
\begin{cases}
\psi_j(x), & x \in U, \\[0.5ex]
\min\{j,u(x)\}, & x \in \R^{n} \setminus U.
\end{cases}
\]
Clearly $g_j \in W^{s,p}(U) \cap  L^{p-1}_{sp}(\R^n)$ by smoothness of $\psi_j$ and the fact that $u_- \in L^{p-1}_{sp}(\R^n)$. Now we can solve the obstacle problem using the functions $g_j$ as obstacles to obtain solutions $u_j \in \mathcal K_{g_j,g_j}(D,U)$, $j=1,2,\dots$, so that $u_j$ is continuous in $\overline{D}$ by Theorem~\ref{thm:cont up to bdry} and a weak supersolution in $D$ by Theorem~\ref{obst prob sol}.
To see that $\{u_j\}$ is an increasing sequence, denote by $A_j := D \cap \{u_j > g_j\}$.  
Since $u_j$ is a weak solution in $A_j$ by Corollary~\ref{obst prob free} and clearly $u_{j+1} \geq u_j$ in $\R^n \setminus A_j$, the comparison principle (Lemma~\ref{comp principle}) implies that $u_{j+1} \geq u_j$.
Similarly, $u_j \leq u$ by Definition~\ref{def_superharmonic}(iii). Since $g_j$ converges pointwise to $u$, we must also have that 
\[
\lim_{j \to \infty} u_j(x) = u(x) \qquad \text{for all } x \in \R^n.
\]
This finishes the proof.
\end{proof}

Below we will show that, as expected, an $(s,p)$-superharmonic function bounded from above is a weak supersolution to~\eqref{problema}. This proves the first statement of Theorem~\ref{thm:superharmonic}(iv).

\begin{theorem}\label{thm_supersuper}
Let $u \in L^{p-1}_{sp}(\R^n)$
be an $(s,p)$-superharmonic function in $\Omega$ that is locally bounded from above in $\Omega$. Then $u$ is a weak supersolution in $\Omega$.
\end{theorem}
\begin{proof} 
Let $D \Subset \Omega$ be an open set such that $\R^n \setminus D$ satisfies the measure density condition~\eqref{eq:dens cond}.
Then by Lemma~\ref{lem_approximation2} there is an increasing sequence $\{u_j\}$ of weak supersolutions in $D$ converging to $u$ pointwise in $\R^n$ such that each $u_j$ is continuous in $\overline D$.
Since each $u_j$ satisfies $u_1 \le u_j \le u$ with $u_1, u \in L^{p-1}_{sp}(\R^n)$ and $u$ is bounded from above in $D$, $u$ is a weak supersolution in $D$ by Theorem~\ref{lemma:conv supersolution}. Finally, because of the arbitrariness of the set $D \Subset \Omega$, we can deduce that the function $u$ is a weak supersolution in $\Omega$, as desired.
\end{proof}

If an $(s,p)$-superharmonic function is a fractional Sobolev function, it is a weak supersolution as well. This gives the second statement of Theorem~\ref{thm:superharmonic}(iv).

\begin{corollary} \label{cor_supersuper}
Let $u \in W^{s,p}_{\rm loc}(\Omega) \cap L^{p-1}_{sp}(\R^n)$ be an $(s,p)$-superharmonic function in $\Omega$.
Then $u$ is a weak supersolution 
in $\Omega$.
\end{corollary}
\begin{proof}
For any $M>0$, denote by $u_M:=\min\{u,M\}$, which is $(s,p)$-superharmonic in $\Omega$ as a pointwise minimum of two $(s,p)$-superharmonic functions. By Theorem~\ref{thm_supersuper} $u_M$ is a weak supersolution in $\Omega$.
Consequently, Corollary~\ref{lemma:conv supersolution2} yields that $u$ is a weak supersolution in $\Omega$.
\end{proof}

On the other hand, lower semicontinuous representatives of weak supersolutions are $(s,p)$-superharmonic.
 
\begin{theorem}\label{thm_supersuper2}
Let $u$ be a lower semicontinuous weak supersolution  
in~$\Omega$ satisfying
\begin{equation} \label{eq:essliminf000}
u(x)=\essliminf_{y \to x}u(y) \qquad \text{for every } x\in\Omega.
\end{equation}
Then $u$ is an $(s,p)$-superharmonic function in $\Omega$.
\end{theorem}
\begin{proof}
According to the definition of $u$, by Lemma \ref{l.tail in control} and Lemma~\ref{hM bound}, together with~\eqref{eq:essliminf000}, we have that (i--ii) and (iv) of Definition~\ref{def_superharmonic} hold. Thus it remains to check that $u$ satisfies the comparison given in Definition~\ref{def_superharmonic}(iii).
For this, take $D \Subset\Omega$ and a weak solution $v$ in $D $ such that $v\in C(\overline D )$, $v \leq u$ almost everywhere in $\R^n\setminus D $ and $v \leq u$ on $\partial D $. For any $\eps>0$ define $v_\eps:=v-\eps$ and consider the set $K_\eps=\big\{ v_\eps \geq u \big \} \cap \overline D $. Notice that by construction the set $K_\eps$ is compact and $K_\eps \cap \partial  D=\emptyset$. Thus, it suffices to prove that $K_\eps=\emptyset$. This is now a plain consequence of the comparison principle proven in Section~\ref{sec_most}. Indeed, one can find an open set $D_1$ such that $ K_\eps\subset D_1  \Subset D$. Moreover, $v_\eps \leq u$ in $\R^n\setminus D_1$ almost everywhere and thus Corollary~\ref{comp principle2} yields $u\geq v_\eps$ almost everywhere in $D_1$. In particular, $u \geq v-\eps$ almost everywhere in $ D$. To obtain an inequality that holds everywhere in $ D$, fix $x\in D$. Then there exists $r>0$ such that $B_r(x)\subset D$ and
\[
u(x) \geq \essinf_{B_r(x)}u-\eps \geq \inf_{B_r(x)}v-2\,\eps \geq v(x)-3\,\eps,
\]
by \eqref{eq:essliminf000} and continuity of $v$.
Since $\eps>0$ and $x\in D$ were arbitrary, we have $u \geq v$ in $ D$. This finishes the proof.
\end{proof}

From Theorem~\ref{thm_supersuper} and Theorem~\ref{thm_supersuper2} we see that a function is a continuous weak solution in $\Omega$ if and only if it is both $(s,p)$-superharmonic and $(s,p)$-subharmonic in~$\Omega$. 
\begin{corollary} \label{cor_harmharm}
A function $u$ is $(s,p)$-harmonic in $\Omega$ if and only if $u$ is a continuous weak solution in $\Omega$.
\end{corollary}

\subsection{Pointwise behavior}

We next investigate the pointwise behavior of $(s,p)$-superharmonic functions in $\Omega$ and start with the following lemma.

\begin{lemma} \label{lemma:pointwise 0}
Let $u$ be $(s,p)$-superharmonic in $\Omega$ such that $u = 0$ almost everywhere in $\Omega$.
Then $u=0$ in $\Omega$.
\end{lemma}
\begin{proof}
Since $u$ is lower semicontinuous, we have $u \leq 0$ in $\Omega$.
Furthermore, we can assume that $u \leq 0$ in the whole $\R^n$ by considering the $(s,p)$-superharmonic function $\min\{u,0\}$ instead of $u$.
Let $z\in\Omega$ and take $R>0$ such that $B_{R}(z) \Subset \Omega$.
By Lemma~\ref{lem_approximation2} there is an increasing sequence $\{u_j\}$ of weak supersolutions in $B_{R}(z)$ converging to $u$ pointwise in $\R^n$ such that each $u_j$ is continuous in $\overline B_{R}(z)$.
Then it holds, in particular, that $u_j(z) \leq u(z)$.
Thus, it suffices to show that for every $\eps>0$ there exists $j$ such that $u_j(z) \geq -\eps$.
To this end, let $\eps>0$.
Since $-u_j$ is a weak subsolution in $B_{2r}(z)$ for any $r\leq R/2$,
applying Theorem~\ref{thm_local} with $\delta=1$ we have that
\begin{align} \label{supBrz-uj}
\sup_{B_{r}(z)}(-u_j) &\leq c\left(\mean{B_{2r}} (-u_j)_+^p\dx\right)^{\frac 1p} + {\rm Tail}((-u_j)_+;z,r) \nonumber \\
&\leq c\left(\mean{B_{2r}} |u_j|^p\dx\right)^{\frac 1p} + c\left(r^{sp}\int_{B_R\setminus B_r}|u_j(y)|^{p-1}|z-y|^{-n-sp}\dy\right)^{\frac 1{p-1}} \nonumber \\
&\quad +c\left(r^{sp}\int_{\R^{n}\setminus B_{R}}|u_j(y)|^{p-1}|z-y|^{-n-sp}\dy\right)^{\frac 1{p-1}} \nonumber \\
&\leq c\left(\mean{B_{2r}} |u_j|^p\dx\right)^{\frac 1p} + c\left(r^{sp}\int_{B_R\setminus B_r}|u_j(y)|^{p-1}|z-y|^{-n-sp}\dy\right)^{\frac 1{p-1}} \\
&\quad +c\left(\frac rR\right)^{\frac{sp}{p-1}}{\rm Tail}(u_1;z,R). \nonumber 
\end{align}
Now, we first choose $r$ to be so small that the last term on the right-hand side of~\eqref{supBrz-uj} is smaller than $\eps/3$.
Then we can choose $j$ so large that each of the two first terms on the right-hand side of~\eqref{supBrz-uj} is smaller than $\eps/3$.
This is possible according to the dominated convergence theorem since $u_j \to 0$ almost everywhere in $B_R(z)$ as $j \to \infty$ and $|u_j| \leq |u_1|$ for every $j$.
Consequently, $u_j(z) \geq -\eps$ and the proof is complete.
\end{proof}

An $(s,p)$-superharmonic function has to coincide with its inferior limits  in~$\Omega$.
In particular, the function cannot have isolated smaller values in single points. This gives Theorem~\ref{thm:superharmonic}(i).

\begin{theorem} \label{thm:essliminf}
Let $u$ be $(s,p)$-superharmonic in $\Omega$. Then
\[
u(x)= \liminf_{y \to x} u(y) = \essliminf_{y \to x}u(y) \quad \text{for every } x\in\Omega.
\]
In particular, $\inf_D u = \essinf_D u$ for any open set $D \Subset \Omega$. 
\end{theorem}
\begin{proof}
Fix $x\in\Omega$ and denote by $\lambda:=\essliminf_{y\to x}u(y)$. Then
\[
\lambda \geq \liminf_{y\to x}u(y) \geq u(x)
\]
by the lower semicontinuity of $u$. To prove the reverse inequality, pick $t<\lambda$.
Then there exists $r>0$ such that $B_r(x) \subset \Omega$ and $u \geq t$ almost everywhere in $B_r(x)$.
By Lemma~\ref{lemma:pointwise 0} the $(s,p)$-superharmonic function
\[
v:=\min\{u,t\}-t
\]
is identically $0$ in $B_r(x)$. In particular, $u(x) \geq t$ and the claim follows by arbitrariness of $t<\lambda$.
\end{proof}

\vs\subsection{Summability of $(s,p)$-superharmonic functions}
We recall a basic result from~\cite[Lemma 7.3]{KMS15}, which is in turn based on the Caccioppoli inequality and the weak Harnack estimates for weak supersolutions presented in~\cite{DKP14}. 
In \cite{KMS15} it is given for equations involving nonnegative source terms, but the proof is identical in the case of weak supersolutions. The needed information is that the weak supersolution belongs locally to $W^{s,p}$.  

\begin{lemma} \label{lemma:lower 2}
Let $u$ be a nonnegative weak  supersolution in $B_{4r} \equiv B_{4r}(x_0)\subset \Omega$.  
Let $h \in (0,s)$, $q \in (0,\bar q)$, where 
\begin{equation}  \label{e.bar q}
\bar q := \min\left\{\frac{n(p-1)}{n-s},p\right\}.
\end{equation}
Then there exists a constant $c \equiv c(n,p,s,\Lambda,s-h,\bar q - q)$ such that 
\begin{equation*} 
\left(\int_{B_{2r}} \mean{B_{2r}} \frac{|u(x)-u(y)|^{q}}{|x-y|^{n+hq}} \dxy \right)^{\frac 1q} \leq \frac{c}{ r^{h}}  \left(\essinf_{B_{r}}  u + \text{\rm Tail}(u_-; x_0, 4r)\right)
\end{equation*}
holds.
\end{lemma}

The next theorem tells that the positive part of an $(s,p)$-superharmonic function also belongs to the Tail space and describes summability properties of solutions, giving Theorem~\ref{thm:superharmonic}(ii).

\begin{theorem} \label{t.summability}
Suppose that $u$ is an $(s,p)$-superharmonic function in $B_{2r}(x_0)$. Then $u \in L_{sp}^{p-1}(\R^n)$. Moreover, defining the quantity 
\begin{equation*}  
M:= \sup_{z \in B_{r}(x_0)} \left( \inf_{B_{r/8}(z)} u_+ + {\rm Tail}(u_-;z,r/2) + \sup_{B_{3r/2}(x_0)} u_- \right),  
\end{equation*}
then $M$ is finite and for $h \in (0,s)$, $q \in (0,\bar q)$ and $t \in (0, \bar t)$, where $\bar q$ is as in~\eqref{e.bar q} and $\bar t$ as in~\eqref{e.bar t}, there is a positive finite constant $C \equiv C(n,p,s,\Lambda,s-h,\bar q-q,\bar t - t)$ such that 
\begin{equation} \label{e.WhqLt bound}
r^h \left[u\right]_{W^{h,q}(B_r(x_0))} + \|u\|_{L^{t}(B_{r}(x_0))} \leq C  M.
\end{equation}
\end{theorem}
\begin{proof}
First, $M$ is finite  
due to assumptions (i) and (iv) of Definition~\ref{def_superharmonic}.
Since $u$ is locally bounded from below, we may assume, without loss of generality, that $u$ is nonnegative in $B_{3r/2}(x_0)$.
Let $u_k := \min\{u,k\}$, $k \in \mathbb{N}$. By Theorem~\ref{thm_supersuper} we have that $u_k$ is a lower semicontinuous weak supersolution in $B_{2r}(x_0)$. 
Let $z \in B_r(x_0)$. The weak Harnack estimate (Theorem~\ref{thm_weakharnack}) for $u_k$ together with Fatou's Lemma, after letting $k \to \infty$,  then imply that
\begin{equation} \label{e.Lt bound}
r^{\frac nt}\|u\|_{L^t(B_{r/4}(z))} \leq c \inf_{B_{r/2}(z)}u + c\,{\rm Tail}(u_-;z,r/2)
\end{equation}
for any $t \in (0, \bar t)$. 
Similarly, Lemma~\ref{lemma:lower 2} applies for $u_k$, and we deduce from it, by Fatou's Lemma  
that 
\begin{equation} \label{e.Whq bound}
r^{h+\frac nq}[u]_{W^{h,q}(B_{r/4}(z))} \leq c \inf_{B_{r/8}(z)}u + c\,{\rm Tail}(u_-;z,r/2)
\end{equation}
for any $h \in (0,s)$ and $q \in (0,\bar q)$.
Now \eqref{e.WhqLt bound} follows from \eqref{e.Lt bound} and \eqref{e.Whq bound} after a covering argument.
Finally, Lemma~\ref{l.tail in control} implies that $u \in L_{sp}^{p-1}(\R^n)$  
from the boundedness of $\left[u\right]_{W^{h,q}(B_r(x_0))}$ and $\|u\|_{L^{t}(B_{r}(x_0))}$ when taking $t=q=p-1$. 
\end{proof}

\vs\subsection{Convergence properties}
We next collect some convergence results related to $(s,p)$-superharmonic functions. 
The first one is that the limit of an increasing sequence of $(s,p)$-super-\break harmonic 
 functions in an open set $\Omega$ is either identically $+\infty$ or $(s,p)$-super-\break harmonic 
  in $\Omega$. Observe that $\Omega$ does not need to be a connected set which is in strict contrast with respect to the local setting. 

\begin{lemma} \label{l.increasing}
Let $\{u_k\}$ be an increasing sequence of $(s,p)$-superharmonic functions in an open set $\Omega$ converging pointwise to a function $u$ as $k \to \infty$. Then either $u \equiv + \infty$ in $\Omega$ or $u$ is $(s,p)$-superharmonic in $\Omega$.
\end{lemma}

\begin{proof}
Observe that since $u \geq u_1$ and $(u_1)_- \in L_{sp}^{p-1}(\R^n)$ by Definition~\ref{def_superharmonic}(iv), we also have that $u_- \in L_{sp}^{p-1}(\R^n)$.

\smallskip
\emph{Step 1.} Assume first that there is an open set $D \subset \Omega $ such that $u$ is finite almost everywhere in $D$. Then we clearly have that $u$ satisfies (i--ii), (iv) of Definition~\ref{def_superharmonic} in $D$. Thus we have to check Definition~\ref{def_superharmonic}(iii). 
Let $D_4 \Subset D$ and let $v$ be as in Definition~\ref{def_superharmonic}(iii) (with $D\equiv D_4$), i.~\!e., $v \in C(\overline{D}_4)$ is a weak solution in $D_4$ such that $v_+ \in L^\infty(\R^n)$ and $v \leq u$ on $\partial D_4$ and almost everywhere on $\R^n \setminus D_4$. For any $\eps>0$, by the lower semicontinuity of $u-v$, there are open sets $D_1, D_2, D_3$ such that $D_1 \Subset  D_2 \Subset D_3 \Subset D_4$, $\R^n \setminus D_2$ satisfies the measure density condition \eqref{eq:dens cond}, and $\{u \leq v - \eps\} \cap D_4 \subset D_1$.
In particular, $u > v - \eps$ on $D_4 \setminus D_1$ and almost everywhere on $\R^n \setminus D_4$.
Since $\overline{D}_3 \Subset D_4$ we have by the compactness that there is large enough $k_\eps$ such that $\overline{D}_3 \setminus  D_2 \Subset \{u_k > v - \eps\} $ for $k>k_\eps$. Indeed, since
\[
\{u>v-\eps\} \cap D_4 = \bigcup_{k} \{u_k>v-\eps\} \cap D_4,
\]
we have that $\big\{ \{u_k>v-\eps\} \cap D_4 \big\}_k$ is an open cover for the compact set $\overline{D}_3 \setminus  D_2$. Defining $\widetilde u_k = v - \eps$ on $D_3 \setminus D_2$ and $\widetilde u_k = \min\{u_k,v-\eps\}$ on $\R^n \setminus D_3$, we have by Lemma~\ref{l.stability} below (applied with $\Omega \equiv D_3$, $D\equiv D_2$, $u_k \equiv \widetilde u_k$) that there is a sequence of weak solutions $\{v_k\}$ in $D_2$ such that $v_k \in C(\overline D_2)$, $v_k \to v-\eps$ in $D_2$ and almost everywhere in $\R^n \setminus D_2$, and that $v_k \leq u_k$ on $\partial D_2$ and almost everywhere on $\R^n \setminus D_2$ whenever $k>k_\eps$. Therefore, by Definition~\ref{def_superharmonic}(iii), $u_k \geq v_k$ in $D_2$ as well. Since the convergence of $v_k \to v-\eps$ is uniform in $\overline{D}_1$ by Arzel\`a--Ascoli Theorem as $k \to \infty$, we obtain that $u \geq v -2\eps$ in $D_1$, and therefore also in the whole $D_4$. This shows that $u$ is $(s,p)$-superharmonic in $D$. 

\smallskip
\emph{Step 2.} Let us next assume that $u$ is not finite on a Borel subset $E$ of $\Omega$ having positive measure. Using inner regularity of the Lebesgue measure we find a compact set $K \subset \Omega$ with positive measure such that $u = +\infty$ on $K$. Then there has to be a ball $B_{r}(x_0)$ such that $| K \cap B_{r}(x_0)| > 0$ and $B_{2r}(x_0) \Subset \Omega$. In particular, for nonnegative $(s,p)$-superharmonic functions defined as $w_k := u_k - \inf_{B_{2r}(x_0)} u_1$, $k \in \mathbb{N}$, we have by the monotone convergence theorem that $\| w_k  \|_{L^{p-1}(B_r(x_0))} \to + \infty$ as $k \to \infty$. Then Theorem~\ref{t.summability} implies that $\inf_{B_\rho(z)} w_k \to +\infty$ as $k \to \infty$ for some smaller ball $B_\rho(z)$ and that $u \equiv + \infty$ in $B_\rho(z)$.  This also implies that $u \notin L_{sp}^{p-1}(\R^n)$. 

\smallskip
\emph{Step 3.} Conclusion. If there is \emph{any} non-empty open set $D \Subset \Omega $ such that $u$ is finite almost everywhere in $D$, then Step 1 yields that $u$ is $(s,p)$-superharmonic in $D$. Therefore Theorem~\ref{t.summability} implies that in fact $u \in L_{sp}^{p-1}(\R^n)$. By Step 2 this excludes the possibility of having a Borel subset $E$ of $\Omega$ with positive measure such that $u$ is not finite on $E$.  Suppose now that there is $E$ as in Step 2. The only possibility that this situation occurs is that \emph{every} ball $B_r(z)$ such that $B_{2r}(z) \Subset \Omega $ contains a Borel set $E_{z,r}$ with positive measure such that $u$ is not finite on $E_{z,r}$ (otherwise $B_r(z)$ would work as $D$).
Step~2~then implies that $\inf_{B_r(z)} u = + \infty$, and hence either $u \equiv + \infty$ in $\Omega$ or $u$ is finite almost everywhere in $\Omega$, implying that $u$ is $(s,p)$-superharmonic in $\Omega$ by Step 1. 
\end{proof}

In the  proof above we appealed to the following stability result. 

\begin{lemma} \label{l.stability}
Suppose that $v$ is a continuous weak solution in $\Omega$ and let $D \Subset \Omega$ be an open set such that $\R^n \setminus D$ satisfies the measure density condition~\eqref{eq:dens cond}. Assume further that there are $h,g \in L_{sp}^{p-1}(\R^n)$ and a sequence 
$\{ u_k \}$ such that $h \leq u_k \leq g$ and $u_k \to v$ almost everywhere in $\R^n \setminus \Omega$ as $k \to \infty$.
Then there is a sequence of weak solutions $\{v_k\}$ in $D$ such that $v_k \in C(\overline D)$, $v_k = v$ on $\Omega \setminus D$, $v_k = u_k$ on $\R^n \setminus \Omega$,  and  $v_k \to v$ everywhere in $D$ and almost everywhere on $\R^n \setminus D$ as $k \to \infty$. 
\end{lemma}

\begin{proof}
Let $U$ be such that $D \Subset U \Subset \Omega$ and $v \in W^{s,p}(U)$. 
Setting $g_k := v$ on $\Omega$ and $g_k = u_k$ on  $\R^n \setminus \Omega$, we find by Corollary~\ref{obst prob free} functions $\{v_k\}$, $v_k \in \mathcal K_{g_k,-\infty}(D,U)$, as in the statement. Indeed, $g_k \in W^{s,p}(U) \cap L_{sp}^{p-1}(\R^n)$.
We may test the weak formulation of $v_k$ with $\phi_k :=(v_k-v)\chi_{U} \in W_0^{s,p}(D)$ and obtain after straightforward manipulations (see e.g. proof of \cite[Lemma 3]{KKP15}) that, for a universal constant $C$,
\begin{equation} \label{e.vkWspU uniform}
\left \| v_k \right\|_{W^{s,p}(U)}  \leq  C \left\| v \right\|_{W^{s,p}(U)}  + C \left( \int_{\R^n \setminus D} \frac{(|h(x)|+|g(x)|)^{p-1}}{(1 + |x|)^{n+sp}} \dx\right)^{\frac1{p-1}}.
\end{equation}
Therefore the sequence $\{v_k\}$ is uniformly bounded in $W^{s,p}(U)$, and the precompactness of $W^{s,p}(U)$, as shown for instance in~\cite[Theorem 7.1]{DPV12}, guarantees that there is a subsequence $\{v_{k_j}\}_j$ converging almost everywhere to $\widetilde v$ as $j \to \infty$. By Corollary~\ref{harnack conv 0} the convergence is pointwise in $D$ and $\widetilde v$ is $(s,p)$-harmonic in $D$. We will show that actually $\widetilde v = v$ in $D$. Since every subsequence of $\{v_k\}$ has such a subsequence, we have that $\lim_{k \to \infty} v_k = v$ pointwise in $D$.  

To see that $v=\widetilde v$ in $D$, we test the weak formulation with $\eta_k :=(v-v_k)\chi_{U} \in W_0^{s,p}(D)$, relabeling the subsequence. Notice that $\eta_k$ is a feasible test function since $v,v_k \in W^{s,p}(U)$ and $v_k=v$ in $U \setminus D$. The weak formulation for $v$ and $v_k$ gives
\begin{align*}
0 &= \int_{\R^n}\int_{\R^n}\big(L(v(x),v(y))-L(v_k(x),v_k(y))\big)\big(\eta(x)-\eta(y)\big)K(x,y)\dxy \nonumber \\[1ex]
&= \int_{U}\int_{U}\big(L(v(x),v(y))-L(v_k(x),v_k(y))\big) \nonumber \\[1ex]
&\qquad\qquad \times \big(v(x)-v(y)-v_k(x)+v_k(y)\big)K(x,y)\dxy \nonumber \\
&\quad + 2 \int_{\R^n \setminus U}\int_{U}\big(L(v(x),v(y))-L(v_k(x),v_k(y))\big)  \big(v(x)-v_k(x)\big)K(x,y)\dxy \nonumber \\[1ex]
&=: I_{1,k} + 2I_{2,k}.  \nonumber 
\end{align*}
We claim that $\lim_{k \to \infty} I_{2,k} = 0$. Indeed, noticing that since $v_k(x) = v(x) $ for $x \in U\setminus D$, we may rewrite 
\begin{equation*} 
I_{2,k} = \int_{\R^n \setminus U}\int_{D}\big(L(v(x),v(y))-L(v_k(x),v_k(y))\big)  \big(v(x)-v_k(x)\big)K(x,y)\dxy.
\end{equation*}
The involved measure  $K(x,y)\dxy$ is finite on $D \times \R^n \setminus U$ and thus we have by the dominated convergence theorem, using the uniform bounds $h \leq u_k \leq g$, the estimate in~\eqref{e.vkWspU uniform}, 
and the fact that $\widetilde v = v$ almost everywhere on $\R^n \setminus D$, that 
\begin{equation*}  
\lim_{k \to \infty} I_{2,k} = \int_{\R^n \setminus U}\int_{D}\big(L(v(x),v(y))-L(\widetilde v(x),v (y))\big)  \big(v(x)-\widetilde v(x)\big)K(x,y)\dxy.
\end{equation*}
Therefore $\lim_{k \to \infty} I_{2,k} \geq 0$ by the monotonicity of $t \mapsto L(t,v(y))$. Thus, Fatou's Lemma implies that 
\begin{align*}  
0  \geq \liminf_{k\to \infty} I_{1,k}
& \geq \int_{U}\int_{U}\big(L(v(x),v(y))-L(\widetilde v(x),\widetilde v(y))\big) \nonumber \\
&\qquad\qquad \times \big(v(x)-v(y)-\widetilde v(x)+\widetilde v(y)\big)K(x,y)\dxy \nonumber,
\end{align*}
proving by the monotonicity of $L$ that $\widetilde v = v$ almost everywhere. This finishes the proof. 
\end{proof}

We also get a fundamental convergence result for increasing sequences of $(s,p)$-harmonic functions, improving Corollary~\ref{harnack conv 0}.

\begin{theorem}[{\bf Harnack's convergence theorem}] \label{harnack conv}
Let $\{u_k\}$ be an increasing sequence of $(s,p)$-harmonic functions in $\Omega$ converging pointwise to a function $u$ as $k \to \infty$.
Then either $u\equiv + \infty$ in $\Omega$ or $u$ is $(s,p)$-harmonic in~$\Omega$. 
\end{theorem}
\begin{proof}
By Lemma~\ref{l.increasing} either $u \equiv + \infty$ or $u$ is $(s,p)$-superharmonic in $\Omega$. In the latter case, Theorem~\ref{t.summability} implies that $u \in L_{sp}^{p-1}(\R^n)$, and thus by Corollary~\ref{harnack conv 0} together with Corollary \ref{cor_harmharm}, $u$ is $(s,p)$-harmonic in $\Omega$. 
\end{proof}

\vs\subsection{Unbounded comparison} 
In Definition~\ref{def_superharmonic}(iii) we demanded that the comparison functions are globally bounded from above. A reasonable question is then that how would the definition change if one removes this assumption. In other words, if the solution is allowed to have too wild nonlocal contributions, would this be able to break the comparison? The answer is negative.  Indeed, the next lemma tells that one can remove the boundedness assumption $ v_+ \in L^\infty(\R^n)$ in the definition of $(s,p)$-superharmonic functions and still get the same class of functions. This is Theorem~\ref{thm:superharmonic}(iii). 

\begin{lemma} \label{l.(iii) vs (iii')}
Let $u$ be an $(s,p)$-superharmonic function in $\Omega$. Then it satisfies the following unbounded comparison statement: 
\begin{itemize}
\item[(iii')] $u$ satisfies the comparison in $\Omega$ against solutions, that is, if $D \Subset \Omega$ is an open set and $v \in C(\overline{D})$ is a weak solution in $D$ such that $u \geq v$ on $\partial D$ and almost everywhere on~$\R^n \setminus D$, then $u \geq v$ in $D$.
\end{itemize}
\end{lemma}

\begin{proof}
Let $u$ be an $(s,p)$-superharmonic function in $\Omega$. We will show that then it also satisfies (iii'). To this end, take $D \Subset \Omega$ and $v$ as in (iii'). 
Let $\eps>0$. Due to lower semicontinuity of $u-v$ and the boundary condition, the set $K_\eps := \{ u \leq v - \eps\} \cap D$ is a compact set of $D$. Therefore we find open sets $D_1,D_2$ such that $K_\eps \subset D_1 \Subset D_2 \Subset D$ and $\R^n \setminus D_2$ satisfies the measure density condition \eqref{eq:dens cond}. Truncate $v$ as $u_k := \min\{v - \eps,k\}$. Applying Lemma~\ref{l.stability} 
(with $\Omega \equiv D$ and $D \equiv D_2$) we find a sequence of continuous weak solutions $\{v_k\}$ in~$ D_2$ such that $v_k \to v-\eps$ in~$D_2$. The convergence is uniform in~$\overline{D}_1$. Therefore, there is large enough $k$ such that $|v_k - v| \leq 2 \eps$ on~$\overline{D}_1$. Moreover, by the comparison principle (Lemma \ref{comp principle}), $v_k \leq v$ in $\R^n$. Since $u > v_k - \eps$ on $\partial D_1$ and almost everywhere in $\R^n \setminus D_1$ by the definition of $K_\eps$, we have by Definition~\ref{def_superharmonic}(iii) that $u \geq v_k - \eps \geq v - 3 \eps$ in $D_1$, and thus we also have that 
$u \geq v - 3 \eps$ in the whole $D$, because in $D \setminus K_\eps$ we have $u > v - \eps$. Since this holds for an arbitrary positive $\eps$, we have that (iii') holds, completing the proof.
\end{proof}

We conclude the section by a more general version of the comparison principle. 
\begin{theorem}[{\bf Comparison principle}] \label{thm:comparison}
Let $u$ be $(s,p)$-superharmonic in $\Omega$ and let $v$ be $(s,p)$-subharmonic in $\Omega$.
If $u \geq v$ almost everywhere in $\R^n \setminus \Omega$ and
\[
\liminf_{\Omega \ni y \to x} u(y) \geq \limsup_{\Omega \ni y \to x} v(y) \qquad \text{for all } x \in \partial\Omega
\]
such that both sides are not simultaneously $+\infty$ or $-\infty$, then $u \geq v$ in $\Omega$.
\end{theorem}

\begin{proof}
Suppose that $u$ and $v$ satisfy the assumptions of the theorem. Let $\eps>0$.
Then there exists an open set $D \Subset \Omega$ such that $u \geq v-\eps$ in $\Omega \setminus D$ by the boundary condition for $u$ and $v$. We may also assume that $\R^n \setminus D$ satisfies the measure density condition \eqref{eq:dens cond}.
Let $U$ be an open set such that $D \Subset U \Subset \Omega$, and let $\{\psi_j\}$, $\psi_j \in C^\infty(\overline U)$, be an increasing sequence converging pointwise to $u$ in $U$. Such a sequence exists according to Lemma \ref{lem_approximation}.
Then $\psi_j \geq v-2\eps$ in $\overline U \setminus D$ whenever $j$ is large enough by compactness of $\overline U \setminus D$ together with upper semicontinuity of $v-2\eps$.
For such $j$, let $g:=\psi_j\chi_U+u\chi_{\R^n \setminus U}$, which is in $W^{s,p}(U)$ by smoothness of $\psi_j$ and in $L^{p-1}_{sp}(\R^n)$ since $u \in L^{p-1}_{sp}(\R^n)$ by Theorem \ref{t.summability}. 
Letting now $h \in \mathcal K_{g,-\infty}(D,U)$ solve the related Dirichlet problem, $h \in C(\overline D)$ is a weak solution in $D$ by Corollary~\ref{obst prob free} and Theorem~\ref{thm:cont up to bdry}.
Since $u \geq h \geq v-2\eps$ in $\partial D$ and almost everywhere in $\R^n \setminus D$, we have according to Lemma \ref{l.(iii) vs (iii')} that $u \geq h \geq v-2\eps$ in $D$ as well. Also $u \geq v-\eps$ in $\Omega \setminus D$ by the choice of $D$ in the beginning, and consequently $u \geq v-2\eps$ in $\Omega$. The claim follows by letting $\eps \to 0$.
\end{proof}

\vs\section{The Perron method} \label{sec_perrons}
We now turn our focus on Dirichlet boundary value problems. Collecting some of the tools so far, it is rather straightforward to prove existence results outside of the natural energy classes. For instance, we record the following existence and regularity result, which often in practice turns out to be very useful. 

\begin{theorem} \label{thm:basic Dir}
Let $\Omega\Subset \Omega'$ be bounded open sets, and assume that $\R^n \setminus \Omega$ satisfies the measure density condition~\eqref{eq:dens cond}. Suppose that $g \in C(\Omega') \cap L_{sp}^{p-1}(\R^n)$. Then there is a weak solution in $\Omega$, which is continuous in $\Omega'$ and has boundary values $g$ on $\R^n \setminus \Omega$.  Such a solution is unique. 
\end{theorem}
\begin{proof}
By Lemma~\ref{lem_approximation} there is an increasing sequence $\{\psi_j\}$, $\psi_j \in C^\infty(\Omega') \cap L_{sp}^{p-1}(\R^n)$, such that $\psi_j \to g$ pointwise in $\R^n$ as $j \to \infty$. Solving the Dirichlet boundary value problem we find weak solutions $u_j \in \mathcal{K}_{\psi_j,-\infty}(\Omega,\Omega') \cap C(\Omega')$, $j = 1,2,\ldots,$ and the sequence is increasing. Theorem~\ref{thm:cont up to bdry} (see also Remark~\ref{remark:unifcont}) gives a uniform (in $j$) modulus of continuity for $u_j$'s on compact subsets of $\Omega'$. In particular, $u_j$ is uniformly bounded from above in $\Omega$ and hence Theorem~\ref{harnack conv} gives that $u_j$ converges to an $(s,p)$-harmonic function $u$ in $\Omega$.  Thus $u \in  C(\Omega')$ and it is a weak solution as in the statement. 

The uniqueness follows easily, since if $u_1,u_2$ are two solutions as in the statement, then $\{u_1 \geq u_2 + \eps\}$ is compact set of $\Omega$ for all $\eps>0$, and comparison then yields that $u_1 \leq u_2 + \eps$ in $\Omega$. Since this holds for arbitrarily small positive $\eps$, and we may interchange the roles of $u_1$ and $u_2$, we deduce that $u_1 \equiv u_2$. 
\end{proof}

However, our tools provide a much more general setup for Dirichlet problems, given by the  Perron method. Indeed, we conclude this paper by introducing a natural nonlocal counterpart of the celebrated Perron method in nonlinear Potential Theory, as mentioned in the introduction; recall Definition~\ref{perron sol} there.

\vs\subsection{Poisson modification}\label{sec_poisson}
We start by defining the nonlocal Poisson modification.

\begin{theorem} \label{thm:Poisson mod} 
Let $D \Subset \Omega$ be open sets such that $\R^n \setminus D$ satisfies the measure density condition~\eqref{eq:dens cond}. Let $u$ be $(s,p)$-superharmonic in $\Omega$. Then there is a continuous weak solution $w$ in $D$ such that the function $P_{u,D}$, defined as
\begin{equation*} 
P_{u,D}(x) :=
\begin{cases}
w(x), & x \in  D, \\
u(x), & x \in \R^{n} \setminus D,
\end{cases}
\end{equation*}
is an $(s,p)$-superharmonic function in $\Omega$ satisfying $P_{u,D} \leq u$ everywhere in $\R^n$. The function $P_{u,D}$ is called \,{\rm Poisson modification\!} of $u$ in $D$. 
\end{theorem}

\begin{proof}
Let $U$ be an open set such that $D \Subset U \Subset \Omega$ and $\R^n \setminus U$ satisfies the measure density condition \eqref{eq:dens cond}. Then by Lemma \ref{lem_approximation2} there is an inreasing sequence $\{u_k\}$, $u_k \in W^{s,p}(U) \cap C(U)$,
such that $u_k = u$ outside of $U$ and $u_k$ converges to $u$ pointwise in $\R^n$. Further, we find weak solutions $w_k \in \mathcal{K}_{u_k,-\infty}(D,U) \cap C(U)$ in $D$ by Corollary \ref{obst prob free} and Theorem~\ref{thm:cont up to bdry}. Moreover, $\{w_k\}$ is an increasing sequence by the comparison principle, and we may define $w:=\lim_{k \to \infty} w_k$.
Since $w_k \leq u$ by the comparison property, also $w \leq u$.

In addition, $w$ is a continuous weak solution in $D$ by Theorem \ref{harnack conv} together with Corollary \ref{cor_harmharm}, and lower semicontinuous in $U$ as a limit of an increasing sequence of continuous functions, and thus also in $\Omega$. Furthermore, according to its definition, $w=u$ everywhere on $\R^n \setminus D$. 
It is then clear that it satisfies (i--ii) and (iv) of Definition~\ref{def_superharmonic}.

We next check that $w$ satisfies also Definition~\ref{def_superharmonic}(iii). Let $E \Subset \Omega$ be open and let $v \in C(\overline E)$ be a weak solution in $E$ bounded from above such that $v \leq w$ on $\partial E$ and almost everywhere on $\R^n \setminus E$. Since $w\leq u$ in $\R^n$, we have that $v \leq u$ on $\partial E$ and almost everywhere on $\R^n \setminus E$, and thus by Definition~\ref{def_superharmonic}(iii) that $v \leq u$ in $E$ as well. This implies that, since $w = u$ on $\partial D$, we have that $v \leq w$ on  $\partial (E\cap D)$ and almost everywhere on $\R^n \setminus (E \cap D)$.
Now, for any $x \in \partial (E \cap D)$ we have
\[
\liminf_{E \cap D \ni y \to x}w(y) \geq w(x) \geq v(x) = \limsup_{E \cap D \ni y \to x} v(y)
\]
by the lower semicontinuity of $w$ and continuity of $v$ up to the boundary. 
This shows by the comparion principle, Theorem \ref{thm:comparison}, that $w$ satisfies Definition~\ref{def_superharmonic}(iii), and hence $w$ is $(s,p)$-superharmonic in $\Omega$. This finishes the proof since $P_{u,D} \equiv w$.
\end{proof}

The next two lemmas show that there is a natural ordering for the Poisson modifications. 

\begin{lemma} \label{poisson monotone} 
Let $D \Subset \Omega$ be open sets such that $\R^n \setminus D$ satisfies the measure density condition~\eqref{eq:dens cond}. 
Let $u$ and $v$ be $(s,p)$-superharmonic functions in $\Omega$ such that $u \leq v$. Then $P_{u,D} \leq P_{v,D}$ in $\Omega$.
In particular, the Poisson modification of $u$ in $D$ is unique.
\end{lemma}
\begin{proof}
By the proof of Theorem~\ref{thm:Poisson mod}, there is an increasing sequence $\{w_k\}$ converging pointwise to $P_{u,D}$ such that $w_k \in C(\overline D)$ is a weak solution in $D$. Since $P_{v,D} \geq P_{u,D} \geq w_k$ in $\R^n \setminus D$ and $P_{v,D}$ is $(s,p)$-superharmonic in $\Omega$ by Theorem~\ref{thm:Poisson mod}, Lemma \ref{l.(iii) vs (iii')} yields $P_{v,D} \geq w_k$ in $D$. Letting $k \to \infty$ finishes the proof.
\end{proof}

\begin{lemma} \label{poisson monotone set}
Let $D \Subset U \Subset \Omega$ be open sets such that both $\R^n \setminus D$ and $\R^n \setminus U$ satisfy the measure density condition~\eqref{eq:dens cond}. Let $u$ be $(s,p)$-superharmonic in~$\Omega$. Then $P_{u,D} \geq P_{u,U}$ in $\Omega$.
\end{lemma}
\begin{proof}
Since $P_{u,D}=u$ in $\R^n \setminus D$, we have $P_{u,D} \geq P_{u,U}$ in $\R^n \setminus D$ by Theorem~\ref{thm:Poisson mod}.
Moreover, according to Theorem~\ref{thm:Poisson mod} $P_{u,D}$ is $(s,p)$-superharmonic in $\Omega$ and $P_{u,U}$ is a continuous weak solution in $U \supset \overline D$. Thus, Lemma~\ref{l.(iii) vs (iii')} implies $P_{u,D} \geq P_{u,U}$ in $D$.
\end{proof}

\vs\subsection{Perron solutions}\label{sec_perron}

We conclude this paper by considering the Perron solutions we defined in Definition \ref{perron sol}.
The first property is that upper and lower Perron solutions are in order.

\begin{lemma} \label{perron order}
The Perron solutions $\overline H_g$ and $\underline H_g$ satisfy $\overline H_g \geq \underline H_g$ in $\R^n$.
\end{lemma}
\begin{proof}
If $\mathcal U_g$ or $\mathcal L_g$ is empty, there is nothing to prove since $\overline H_g \equiv +\infty$ or $\underline H_g \equiv -\infty$, respectively.
Assume then that the classes are non-empty and take $u \in \mathcal U_g$ and $v \in \mathcal L_g$. Then
\[
\liminf_{\Omega \ni y \to x}u(y) \geq \esslimsup_{\R^n \setminus \Omega \ni y \to x}g(y) \geq \essliminf_{\R^n \setminus \Omega \ni y \to x}g(y) \geq \limsup_{\Omega \ni y \to x}v(y)
\]
for every $x \in \partial\Omega$ by Definition \ref{perron sol}(iii). Both sides of the inequality above cannot be simultaneously $-\infty$ or $+\infty$ according to Definition \ref{perron sol}(ii). Moreover, since $u=g=v$ almost everywhere in $\R^n \setminus \Omega$, we have $u \geq v$ in $\Omega$ by the comparison principle, Theorem \ref{thm:comparison}.
Finally, taking the infimum over $\{u \in \mathcal U_g\}$ and the supremum over $\{v \in \mathcal L_g\}$ finishes the proof.
\end{proof}

The second straightforward observation is that for bounded boundary values the Perron classes are non-empty.

\begin{lemma} \label{l.f bnd}
If $g \in L_{sp}^{p-1}(\R^n)$ is bounded from above, then the class $\mathcal U_g$ is nonempty.
\end{lemma}
\begin{proof}
Let $\sup_{\R^n} g \leq M < \infty$ and take $u:=M\chi_\Omega + g\chi_{\R^n \setminus \Omega}$. Then clearly $u$ satisfies the properties (ii-iv) of Definition \ref{perron sol}. To obtain the property (i), we first have that $u \in W^{s,p}_{\rm loc}(\Omega) \cap L^{p-1}_{sp}(\R^n)$, and testing against a nonnegative test function $\eta \in C^\infty_0(\Omega)$ gives
\begin{align*}
&\int_{\R^n}\int_{\R^n}L(u(x),u(y))\big(\eta(x)-\eta(y)\big)K(x,y)\dxy \\
&\qquad \qquad = \,2\int_{\R^n \setminus \Omega}\int_{\Omega}L(M,g(y))\eta(x)K(x,y)\dxy \, \geq \, 0.
\end{align*}
Thus $u$ is a weak supersolution in $\Omega$, and further $(s,p)$-superharmonic in $\Omega$ by Theorem \ref{thm_supersuper2}.
\end{proof}

Now, we are in a position to prove the main theorem in this section, i.~\!e., Theorem~\ref{thm:Perron} stated in the introduction, which gives the expected alternative result, saying that the Perron solution has to be identically $+\infty$ or $-\infty$, or $(s,p)$-harmonic. 
\begin{proof}[Proof of Theorem \ref{thm:Perron}]
Let us denote by $H_g := \overline H_g$, the case of $H_g := \underline H_g$ being completely analogous. We may assume that $\mathcal U_g$ is non-empty since otherwise $H_g \equiv \infty$ in $\Omega$. Since $\mathcal U_g$ is non-empty, we must have that $(H_g)_+ \in L_{sp}^{p-1}(\R^n)$. According to Choquet's Topological Lemma (see, e.~\!g. \cite[Lemma 8.3]{HKM06}), there exists a decreasing sequence $\{u_j\}$ of functions in $\mathcal U_g$ converging to a function $u$ such that $\widehat u = \widehat H_g$. Here the lower semicontinuous regularization of a function $f \colon \R^n \to [-\infty,\infty]$ is defined by
\begin{equation*} 
\widehat f(x) := \lim_{r \to 0} \inf_{B_r(x)} f.
\end{equation*}
In particular, $\widehat f \leq f$ and $\widehat f$ is lower semicontinuous.
Let $D \Subset  \Omega$ be an open set such that $\R^n \setminus D$ satisfies the measure density condition~\eqref{eq:dens cond}. 
Then $P_{u_j,D} \in \mathcal U_g$ and it is $(s,p)$-harmonic in $D$ for all $j$ by Lemma~\ref{thm:Poisson mod}.
Moreover, $\{P_{u_j,D}\}_j$ is a decreasing sequence by Lemma~\ref{poisson monotone}. Let $v_D := \lim_{j \to \infty} P_{u_j,D}$. By Harnack's convergence theorem (Theorem~\ref{harnack conv}) either $v_D \equiv -\infty$ in $D$ or $v$ is $(s,p)$-harmonic in $D$. Furthermore, we may take any larger $U$ containing $D$ such that $\R^n \setminus U$ satisfies the measure density condition~\eqref{eq:dens cond}.
Since $P_{u_j,U} \leq P_{u_j,D}$ by Lemma \ref{poisson monotone set}, we have that  $\lim_{j \to \infty} P_{u_j,U} \equiv -\infty$ in $U$ if $v_D \equiv -\infty$ in $D$. Thus, $H_g \equiv - \infty$ in $\Omega$ if $v_D \equiv -\infty$ in any regular open component $D$. 

Suppose now that $H_g \not\equiv - \infty$ in $\Omega$. Therefore $v_D$ is 
$(s,p)$-harmonic in $D$ whenever the complement of $D \Subset \Omega$ satisfies the measure density condition~\eqref{eq:dens cond}. Let $D$ be such a set. Theorem~\ref{thm:Poisson mod} yields $H_g \leq P_{u_j,D} \leq u_j$, and taking the limit as $j \to \infty$ and, furthermore, lower semicontinuous regularizations, we obtain
\begin{align} \label{HgvDu}
H_g \leq v_D \leq u \qquad \text{and} \qquad \widehat H_g \leq \widehat v_D \leq \widehat u = \widehat H_g,
\end{align}
respectively. Consequently,
\[
v_D = \widehat v_D = \widehat H_g \leq H_g \leq v_D \quad \text{in } D,
\]
and thus $H_g = v_D$ in $D$.  

To obtain the $(s,p)$-harmonicity for $H_g$ in $\Omega$, let $\{D_k\}$, $k=1,2,\dots$, be an exhaustion of $\Omega$ by open regular subsets such that $D_k \Subset D_{k+1}$.
Proceeding as above, we obtain functions $v_k:= \lim_{j \to \infty}P_{u_j,D_k}$ that are $(s,p)$-harmonic in $D_k$ and $v_k = H_g$ in $D_k$. In particular, we have that $H_g$ is continuous in $\Omega$.
Since $P_{u_j,D_{k+1}} \leq P_{u_j,D_k}$ for every $j$ by Lemma \ref{poisson monotone set}, we have $v_{k+1} \leq v_k$. Let us denote by $v:=\lim_{k \to \infty} v_k$. Then for any $U \Subset \Omega$, $v$ is $(s,p)$-harmonic in $U$ by Theorem \ref{harnack conv} since $v_k$ is $(s,p)$-harmonic in $U$ when $k$ is large enough. The possibility that $v \equiv -\infty$ in $U$ is excluded since $v=H_g \not\equiv -\infty$ in $D_1$. Consequently, $v$ is $(s,p)$-harmonic in the whole $\Omega$. Now
\[
H_g = \widehat H_g = \widehat v_k \geq \widehat v = v \quad \text{in } \Omega
\]
by continuity of $H_g$ and $v$ in $\Omega$ together with \eqref{HgvDu}. The reverse inequality holds by \eqref{HgvDu}, and thus
$H_g = v$ in $\Omega$. Moreover, since $H_g = g = v$ almost everywhere in $\R^n \setminus \Omega$, we conclude that $H_g$ is $(s,p)$-harmonic in $\Omega$.
\end{proof}

\begin{remark}
Lemma \ref{l.f bnd} and Theorem \ref{thm:Perron} do also hold for the more general Perron solutions mentioned in Remark~\ref{perron general}.
\end{remark}

We conclude this paper with the lemma below, which assures that if there is a solution to the Dirichlet problem, then it is necessarily the Perron solution. In particular, this is the case under the natural hypothesis of Theorem~\ref{thm:basic Dir}. 
\begin{lemma}\label{sortofresol}
Assume that $h \in C(\overline\Omega)$ is a weak solution in $\Omega$ such that
\begin{align*}
\lim_{\Omega \ni y \to x} h(y) = g(x) \quad \text{for every } x \in \partial\Omega \qquad \text{and} \qquad h=g \quad \text{a.~\!e. in } \R^n \setminus \Omega
\end{align*}
for some $g \in C(\Omega') \cap L^{p-1}_{sp}(\R^n)$ with $\Omega' \Supset \Omega$. Then $\overline H_g = h = \underline H_g$.
\end{lemma}
\begin{proof}
The situation is symmetric, so we only need to prove the result for $\overline H_g$.
We have $h \geq \overline H_g$ since $h \in \mathcal U_g$.
To obtain the reverse inequality, let $u \in \mathcal U_g$.
Then for every $\eps>0$ there exists an open set $D \Subset \Omega$ such that $u+\eps > h$ in $\R^n \setminus D$.
Consequently, $u+\eps \geq h$ in $D$ since $u+\eps$ is $(s,p)$-superharmonic in $\Omega$.
Letting $\eps \to 0$ we obtain that $u \geq h$, and taking the infimum over $\mathcal U_g$ yields $\overline H_g \geq h$.
\end{proof}

\noindent
\\ {\bf Acknowledgements.}\
We would like to thank the  referees for their  useful suggestions, which allowed to improve the manuscript.

\vspace{3mm}
\end{document}